\documentclass[12pt,reqno]{amsart}

\usepackage{amsthm, amsmath, amsfonts, amssymb, graphicx, tikz-cd, makecell, hyperref, mathtools}
\DeclareMathAlphabet{\mathbbb}{U}{bbold}{m}{n}
\usepackage[capitalize,noabbrev]{cleveref} 
\usetikzlibrary{arrows.meta,calc}
\usepackage[shortlabels]{enumitem}
\usepackage{needspace}

\newtheorem{theorem}{Theorem}[section]
\newtheorem{lemma}[theorem]{Lemma}
\newtheorem{proposition}[theorem]{Proposition}
\newtheorem{corollary}[theorem]{Corollary}

\theoremstyle{definition}
\newtheorem{definition}[theorem]{Definition}
\newtheorem{notation}[theorem]{Notation}
\newtheorem{example}[theorem]{Example}

\newtheorem{remark}[theorem]{Remark}

\newcounter{mycount}

\newcommand{\myref}[1]{\hyperref[#1]{#1}}

\newcommand{\up}{{\uparrow}}

\newcommand{\sfour}{\mathsf{S4}}
\newcommand{\ipc}{\mathsf{IPC}}

\newcommand{\M}{\mathsf{M}}

\newcommand{\K}{\mathbb{K}}
\newcommand{\V}{\mathsf{Var}}

\renewcommand{\H}{\mathsf{H}}
\renewcommand{\S}{\mathsf{S}}
\renewcommand{\P}{\mathsf{P}}
\renewcommand{\L}{\mathsf{L}}
\newcommand{\F}{\mathfrak{F}}
\newcommand{\G}{\mathfrak{G}}

\newcommand{\mipc}{\mathsf{MIPC}}
\newcommand{\mgrz}{\mathsf{MGrz}}

\newcommand{\msfour}{\mathsf{MS4}}

\newcommand{\sfive}{\mathsf{S5}}
\newcommand{\mha}{\mathbb{MHA}}
\newcommand{\ms}{\mathsf{MS4}}
\newcommand{\sa}{\mathbb{S}\mathbbb{4}}
\newcommand{\ha}{\mathbb{HA}}
\newcommand{\qsfour}{\mathsf{QS4}}
\newcommand{\qgrz}{\mathsf{QGrz}}
\newcommand{\msa}{\mathbb{MS}\mathbbb{4}}
\newcommand{\mgrza}{\mathbb{MG}\mathbbb{rz}}
\newcommand{\cpc}{\mathsf{CPC}}
\newcommand{\Op}{\mathcal{O}}

\renewcommand{\L}{\mathsf{L}}
\newcommand{\Lae}{\mathcal{L_{\forall\exists}}}

\newcommand{\Log}{\mathsf{Log}}
\newcommand{\Alg}{\mathsf{Alg}}
\newcommand{\Clop}{\mathsf{Clop}}

\newcommand{\iqc}{\mathsf{IQC}}
\newcommand{\sk}{\rho}
\newcommand{\B}{\mathfrak{B}}
\newcommand{\A}{\mathfrak{A}}
\newcommand{\msfrm}{\mathbf{DF}_{\ms}}
\newcommand{\mipcfrm}{\mathbf{DF}_{\mipc}}
\newcommand{\Grz}{\mathsf{Grz}}
\newcommand{\grz}{\mathsf{grz}}

\newcommand{\bbox}{\blacksquare}

\newcommand\clusterone[2]{
  \path[draw,red] let \p1=(#1)
    in \pgfextra{
    \pgfmathsetmacro{\radius}{#2*0.3}
  }
  (\p1) circle(\radius cm);
}

\newcommand\clustertwo[4]{
  \path[draw,red] let \p1=(#1), \p2=(#2), \p3=($(\p1)!.5!(\p2)$)
  in \pgfextra{
    \pgfmathsetmacro{\angle}{atan2(\y2-\y1,\x2-\x1)}
    \pgfmathsetmacro{\focal}{veclen(\x2-\x1,\y2-\y1)/2/1cm}
    \pgfmathsetmacro{\lentotcm}{\focal*2*#3}
    \pgfmathsetmacro{\axeone}{(\lentotcm - 2 * \focal)/2+\focal}
    \pgfmathsetmacro{\axetwo}{sqrt((\lentotcm/2)*(\lentotcm/2)-\focal*\focal}
    \pgfmathsetmacro{\newaxetwo}{\axetwo*0.5*#4}
  }
  (\p3) ellipse[x radius=\axeone cm,y radius=\newaxetwo cm, rotate=\angle];
}

\setlist[enumerate,1]{label={\upshape(\arabic*)},ref=\arabic*}

\makeatletter 
\edef\plabelformat{(\string#2\string#1\string#3)}
\edef\plabelrangeformat{(\string#3\string#1,\string#2\string#6)}
\newcommand{\plabel}[1]{\label{#1}
\immediate\write\@auxout{\noexpand\crefformat{#1}{\noexpand\cref{#1}\plabelformat}
\noexpand\crefmultiformat{#1}{\noexpand\cref{#1}\plabelformat}{,\plabelformat}{,\plabelformat}{,\plabelformat}
\noexpand\crefrangeformat{#1}{\noexpand\cref{#1}\plabelrangeformat}}}
\makeatother

\makeatletter
\@namedef{subjclassname@2020}{\textup{2020} Mathematics Subject Classification}
\makeatother

\pgfdeclarelayer{edgelayer}
\pgfdeclarelayer{nodelayer}
\pgfsetlayers{edgelayer,nodelayer,main}

\tikzstyle{none}=[inner sep=0pt]
\tikzstyle{black dot}=[fill=black, draw=black, shape=circle, inner sep=0, minimum size=3.5pt]

\tikzstyle{to}=[->]
\tikzstyle{mapsto}=[{|->}]
\tikzstyle{none dashed}=[-, dashed]
\tikzstyle{dashed to}=[->, dashed]
\tikzstyle{dashed mapsto}=[{|->}, dashed]
\tikzstyle{Latex arrow}=[{-{Latex[width=1mm]}}]
\tikzstyle{dashed Latex arrow}=[{-{Latex[width=1mm]}}, dashed]

\begin{document}

\title[Failure of the Blok--Esakia Theorem in the monadic setting]{Failure of the Blok--Esakia Theorem\\ in the monadic setting}

\author{G.~Bezhanishvili}
\address{New Mexico State University}
\email{guram@nmsu.edu}

\author{L.~Carai}
\address{University of Milan}
\email{luca.carai.uni@gmail.com}

\subjclass[2020]{03B45; 03B55; 06D20; 06E25; 06E15}
\keywords{Intuitionistic logic; modal logic; G\"odel translation; Blok--Esakia Theorem; Heyting algebra; Boolean algebra with operators}

\begin{abstract}
The Blok--Esakia Theorem establishes that the lattice of superintuitionistic logics is isomorphic to the lattice of extensions of Grzegorczyk's logic.
We prove that the Blok--Esakia isomorphism $\sigma$ does not extend 
to the fragments of the corresponding predicate logics of already one fixed variable. In other words, we prove that $\sigma$ is no longer an isomorphism from the lattice of extensions of the monadic intuitionistic logic to the lattice of extensions of the monadic Grzegorczyk logic.
\end{abstract}

\maketitle
\tableofcontents

\section{Introduction}

It is a classic result of McKinsey and Tarski \cite{MT48} that the G\"odel translation embeds the intuitionistic propositional calculus $\ipc$ into Lewis' modal system $\sfour$. 
A systematic study of the relationships between extensions of $\ipc$ (superintuitionistic logics) and extensions of $\sfour$ was initiated by Dummett and Lemmon \cite{DL59} and further studied 
by Maksimova and Rybakov \cite{MR74eng}, Blok and Dwinger \cite{BD75}, Blok \cite{Blo76}, and Esakia \cite{Esa76,Esa79,Esa79b}.

Let $\L$ be a superintuitionistic logic. Applying the G\"odel translation $(-)^t$ to the theorems of $\L$ embeds $\L$ into the following normal extension of $\sfour$:
\[
\tau\L = \sfour + \{ \varphi^t : \L \vdash \varphi \}.
\] 
For a normal extension $\M$ of $\sfour$, let 
\[
\rho\M = \{ \varphi : \M \vdash \varphi^t \}.
\] 
Then $\rho\M$ is a superintuitionistic logic, and we call $\M$ a {\em modal companion} of a superintuitionistic logic $\L$ provided $\L = \rho\M$, in which case $\L$ is called the {\em intuitionistic fragment} of $\M$ (see, e.g., \cite[Sec.~9.6]{CZ97}). 

Esakia \cite{Esa79} proved that all modal companions of a superintuitionistic logic $\L$ form the interval $[\tau\L,\sigma\L]$ in the lattice of normal extension of $\sfour$, where $\sigma\L$ is obtained from $\tau\L$ by postulating the {\em Grzegorczyk axiom}
\[
\grz = \Box(\Box(p\to\Box p)\to p)\to p.
\]
In other words, if $\Grz \coloneqq \sfour + \grz$ is the \emph{Grzegorczyk logic}, then
\[
\sigma\L=\Grz + \{ \varphi^t : \L\vdash\varphi \}.
\]

Let $\Lambda(\ipc)$ be the 
lattice of superintuitionistic logics and $\Lambda(\Grz)$ the 
lattice of normal extensions of $\Grz$. 
By the celebrated Blok--Esakia Theorem, $\sigma\colon\Lambda(\ipc)\to\Lambda(\Grz)$ is an isomorphism (see, e.g., \cite[Thm.~9.66]{CZ97}). 

The G\"odel translation has a natural extension to the predicate setting, and Rasiowa and Sikorski (see, e.g., \cite[XI.11.5]{RS63}) proved that it embeds the intuitionistic predicate calculus $\iqc$ into $\qsfour$ (the predicate $\sfour$). However, the behavior of modal predicate companions of superintuitonistic predicate logics 
is much less understood.
For example, it remains open whether the predicate Grzegorczyk logic $\qgrz$ is a modal companion of $\iqc$, let alone the largest modal companion.\footnote{In \cite{Pan89} it is claimed that $\qgrz$ is a modal companion of $\iqc$, and in \cite{Nau91} that it is not the largest modal companion. However, the proofs in \cite{Pan89,Nau91} rely on the Flagg--Friedman translation \cite{FF86} that Inou\'e \cite{Ino92} showed is not faithful.
Therefore, these results require further examination (see \cite[Rem.~2.11.13]{GSS09} and \cite[Rem.~5.16]{BC24b}).}

Hilbert and Ackermann \cite{HA28} initiated the study of the monadic fragment of classical predicate logic, where only 
one variable is allowed in monadic predicates.\footnote{This should not be confused with the monadic fragment, where different variables are allowed in monadic predicates.}
Wajsberg \cite{Waj33} proved that this fragment is axiomatized by $\mathsf{S5}$, and Halmos~\cite{Hal62} conducted an algebraic study of this fragment.
Prior~\cite{Pri57} introduced the monadic intuitionistic calculus $\mipc$,
which Bull \cite{Bul66} proved to axiomatize the monadic fragment of $\iqc$. Fischer Servi \cite{FS77} defined $\msfour$ (monadic $\sfour)$ and proved that the G\"odel translation embeds $\mipc$ into $\msfour$.

Monadic logics are better understood than predicate logics. They can be thought of as bimodal logics \cite{FS77,Esa88,BBI23}, and hence can be studied using the standard semantic tools in modal logic (see, e.g., \cite{GKWZ03}). Because of this, normal extensions of $\mipc$ and $\ms$ have been studied more extensively than their predicate counterparts.
In particular, the monadic Grzegorczyk logic $\mgrz$ was introduced in \cite{Esa88}, where it was shown that $\mgrz$ is a 
modal companion of $\mipc$. It is natural to ask whether the Blok--Esakia Theorem extends to the monadic setting. Our main contribution proves 
that it does {\bf not}. 
Our main tool is the algebraic semantics for $\mipc$ and $\ms$ provided by the varieties $\mha$ of monadic Heyting algebras and $\msa$ of monadic $\sfour$-algebras,
which generalize Halmos' monadic boolean algebras \cite{Hal56}.
We also heavily use the representation theory for $\mha$ and $\msa$, and the corresponding descriptive frames.

The paper is organized as follows.
In \cref{sec:mipc} we recall $\mipc$ and its algebraic and descriptive frame semantics, and in \cref{sec:ms4} we do the same for $\ms$. In \cref{sec: Godel translation} we generalize $\rho$, $\tau$, and $\sigma$ to the monadic setting, and in \cref{sec:Op} we generalize the functor $\Op \colon \sa \to \ha$ associating to each $\sfour$-algebra the Heyting algebra of its open elements to the functor $\Op \colon \msa \to \mha$.
We prove that $\Op$ is the algebraic counterpart of $\rho$ and give the dual description of $\Op$ using the corresponding descriptive frames. 
While $\Op \colon \sa \to \ha$ preserves the class operators $\H$, $\S$, and $\P$ of taking homomorphic images, subalgebras, and products, in \cref{sec: failure BE} we show that $\Op \colon \msa \to \mha$ no longer preserves $\S$.
It is this key observation that allows us to prove that, although $\tau$ and $\sigma$ remain lattice homomorphisms in the monadic setting, $\rho$ is neither a lattice homomorphism nor one-to-one. From this we derive that $\sigma$ is not an isomorphism, thus concluding that the Blok--Esakia Theorem does not extend to the monadic setting.

\section{$\mipc$} \label{sec:mipc}

Let $\mathcal L$ be the propositional language of $\ipc$, and let $\Lae$ be its extension by two ``quantifier modalities" 
$\forall$ and $\exists$. 

\begin{definition} \label{def: MIPC}
The \textit{monadic intuitionistic propositional calculus} $\mipc$ is the smallest set of formulas in the language
$\Lae$ containing
\begin{enumerate}
\item all theorems of 
$\ipc$;
\item the $\sfour$-axioms for $\forall$:
\quad $\forall(p\land q)\leftrightarrow(\forall p\land\forall q)$, \quad $\forall p \rightarrow p$, \quad $\forall p \rightarrow \forall \forall p$;
\item the $\sfive$-axioms for $\exists$:
\quad $\exists(p\vee q)\leftrightarrow(\exists p\vee\exists q)$, \quad $p \rightarrow \exists p$, \quad $\exists \exists p \rightarrow \exists p$,\\ 
\hphantom{the $\sfive$-axioms for $\exists$:} \quad $(\exists p \land \exists q) \rightarrow \exists (\exists p \land q)$;
\item the axioms connecting $\forall$ and $\exists$:
\quad $\exists\forall p\leftrightarrow\forall p$, \quad $\exists p \leftrightarrow \forall\exists p$;
\end{enumerate}
and closed under the rules of modus ponens, substitution, and necessitation $(\varphi / \forall \varphi )$.
\end{definition}

The algebraic semantics for $\mipc$ is provided by monadic Heyting algebras, which were first introduced by Monteiro and Varsavsky \cite{MV57}.

\begin{definition} \label{def:mha}
A \textit{monadic Heyting algebra} is a tuple $\mathfrak A=(H,\forall,\exists)$ such that $H$ is a Heyting algebra and $\forall, \exists$ are unary functions on $H$ satisfying the axioms corresponding to the ones in \cref{def: MIPC}.
\end{definition}

\begin{remark} \label{rem: H0}
For a monadic Heyting algebra $(H,\forall,\exists)$, 
let $H_0=\{ \forall a : a \in H \}$. Using the axioms of monadic Heyting algebras, it is straightforward to check that 
\[
H_0= \{ a \in H : a = \forall a \} = \{ a \in H : a = \exists a \} = \{ \exists a : a \in H \},
\]
and that $H_0$ is a Heyting subalgebra of $H$.
Moreover, 
 $\forall$ is the right adjoint and $\exists$ is the left adjoint of the embedding $H_0\hookrightarrow H$. Furthermore, each monadic Heyting algebra 
is represented as a pair $(H,H_0)$, where the Heyting embedding $H_0\hookrightarrow H$ has both right and left adjoints (see, e.g., \cite[Sec.~3]{Bez98}).
\end{remark}

Clearly the class of monadic Heyting algebras is equationally definable and hence forms a variety. 
We denote the corresponding category 
by $\mha$.

Since terms in the language of monadic Heyting algebras correspond to formulas in $\Lae$, we say that a formula $\varphi$ is \emph{valid} in a monadic Heyting algebra $\mathfrak{A}$ (in symbols $\mathfrak{A} \vDash \varphi$) if the equation $t=1$ holds in $\mathfrak{A}$, where $t$ is the term corresponding to $\varphi$. For a class $\mathbb{K} \subseteq \mha$ of monadic Heyting algebras, we write $\mathbb{K} \vDash \varphi$ if $\mathfrak A \vDash \varphi$ for each $\mathfrak A \in \mathbb{K}$. The standard Lindenbaum-Tarski construction then yields:

\begin{theorem} \cite[Thm.~2]{FS77} \label{thm: alg comp MIPC}
$\mipc \vdash \varphi$ iff $\mha \vDash \varphi$ for each formula $\varphi$ of $\Lae$.
\end{theorem}

\begin{definition}
An \emph{extension $\L$ of $\mipc$} is a set of formulas in the language $\Lae$ containing $\mipc$ and closed under modus ponens, substitution, and necessitation. 
\end{definition}

Each extension $\sf L$ of $\mipc$ 
gives rise to the variety $\Alg(\L)$ of monadic Heyting algebras validating all formulas in $\L$. 
Conversely, each variety $\mathbb{V}$ of monadic Heyting algebras gives rise to the extension $\Log(\mathbb{V})$ of $\mipc$ consisting of the formulas valid in all members of $\mathbb{V}$. By $\Lambda(\mipc)$ we denote the complete lattice of extensions of $\mipc$ and by $\Lambda(\mha)$ the complete lattice of subvarieties of $\mha$. We thus obtain:

\begin{theorem} \cite[Thm.~3]{Bez98} \label{thm: lattice iso}
$\Lambda(\mipc)$ is dually isomorphic to $\Lambda(\mha)$.
\end{theorem}

Esakia duality for Heyting algebras \cite{Esa74,Esa19} was generalized to monadic Heyting algebras in \cite{Bez99}. 
As usual, for a binary relation $R$ on a set $X$ and $S \subseteq X$, we write $R[S]$ for the $R$-image and $R^{-1}[S]$ for the $R$-inverse image of $S$. 
When $S= \{x\}$, we simply write $R[x]$ and $R^{-1}[x]$. We call $S$
an {\em $R$-upset} if $R[S]\subseteq S$ and an {\em $R$-downset} if $R^{-1}[S] \subseteq S$.
If $R$ is a quasi-order (reflexive and transitive relation), we denote by $E_R$ the equivalence relation given by 
\[
x E_R y \ \Longleftrightarrow \ x R y \; \& \; y R x.
\]

A {\em Stone space} is a topological space $X$ that is compact, Hausdorff, and zero-dimensional. We call a binary relation $R$ on $X$ \emph{continuous} if $R[x]$ is closed for each $x \in X$ and $R^{-1}[U]$ is clopen for each clopen $U \subseteq X$.

\begin{definition}\plabel{def:ono}
A \textit{descriptive $\mipc$-frame} is a tuple $\mathfrak F=(X,R,Q)$ such that
\begin{enumerate}
\item\label[def:ono]{def:ono:item1} $X$ is a Stone space,
\item\label[def:ono]{def:ono:item2} $R$ is a continuous partial order,
\item\label[def:ono]{def:ono:item3} $Q$ is a continuous quasi-order,
\item\label[def:ono]{def:ono:item4} 
$U$ a clopen $R$-upset $\Longrightarrow$ $Q[U]$ is a clopen $R$-upset,
\item\label[def:ono]{def:ono:item5} $R \subseteq Q$,
\item\label[def:ono]{def:ono:item6} $x Q y \Longrightarrow \exists z \in X : x R z \; \& \; z E_Q y$.
\end{enumerate}
\begin{figure}[!ht]
\begin{center}
\begin{tikzpicture}
	\begin{pgfonlayer}{nodelayer}
		\node [style=black dot] (1) at (0, 0) {};
		\node [style=black dot] (2) at (0, 2) {};
		\node [style=black dot] (3) at (2, 2) {};
		\node [style=none] (4) at (1, 2.25) {$E_Q$};
		\node [style=none] (5) at (-0.27, 1) {$R$};
		\node [style=none] (6) at (1.4, 0.96) {$Q$};
		\node [style=none] (7) at (0, -0.27) {$x$};
		\node [style=none] (8) at (-0.2, 2.27) {$\exists z$};
		\node [style=none] (9) at (2.25, 2.25) {$y$};
	\end{pgfonlayer}
	\begin{pgfonlayer}{edgelayer}
		\draw [style=dashed Latex arrow] (1) to (2);
		\draw [style=Latex arrow] (1) to (3);
		\draw [style=none dashed] (2) to (3);
	\end{pgfonlayer}
\end{tikzpicture}
\end{center}
\end{figure}
\end{definition}

\begin{remark}
The condition in \cref{def:ono:item6} implies that it is possible to recover the quasi-order $Q$ from $R$ and $E_Q$ in any descriptive $\mipc$-frame. In fact, descriptive $\mipc$-frames can be equivalently presented as triples $(X,R,E)$ where $X$ is a Stone space, $R$ is a quasi-order, and $E$ is an equivalence relation satisfying the conditions corresponding to the ones in \cref{def:ono} (see \cite[Thm.~11(a)]{Bez99}). For our purposes it is more convenient to work with the quasi-order $Q$, but we will employ this different perspective to work with descriptive frames for $\ms$.
\end{remark}

\begin{definition}\plabel{def:mipcfrm-morphisms}
Let $\F_1=(X_1,R_1,Q_1)$ and $\F_2=(X_2,R_2,Q_2)$ be descriptive $\mipc$-frames. A map $f \colon X_1 \to X_2$ is a \emph{morphism of descriptive $\mipc$-frames} if 
\begin{enumerate}
\item\label[def:mipcfrm-morphisms]{def:mipcfrm-morphisms:item1} $f$ is continuous, 
\item\label[def:mipcfrm-morphisms]{def:mipcfrm-morphisms:item2} $R_2[f(x)]=fR_1[x]$ for each $x\in X_1$,
\item\label[def:mipcfrm-morphisms]{def:mipcfrm-morphisms:item3} $Q_2[f(x)]=fQ_1[x]$ for each $x\in X_1$,
\item\label[def:mipcfrm-morphisms]{def:mipcfrm-morphisms:item4} $Q_2^{-1}[f(x)]= R_2^{-1}fQ_1^{-1}[x]$ for each $x\in X_1$.
\end{enumerate}
\end{definition}

\begin{remark}
\cref{def:mipcfrm-morphisms:item2} says that $f$ is a p-morphism with respect to $R$, and \cref{def:mipcfrm-morphisms:item3} that $f$ is a p-morphism with respect to $Q$.
The left-to-right inclusion of \cref{def:mipcfrm-morphisms:item4} follows from \cref{def:mipcfrm-morphisms:item3}, the other inclusion can be expressed as follows: 
\[
z Q_2 f(x) \Longrightarrow \exists y \in X_1 : y Q_1 x \; \& \; z R_2 f(y).
\]
\begin{figure}[!ht]
\begin{center}
\begin{tikzpicture}
	\begin{pgfonlayer}{nodelayer}
		\node [style=none] (4) at (-1, 1) {};
		\node [style=none] (5) at (2.5, 1) {};
		\node [style=none] (6) at (-1, -1.25) {};
		\node [style=none] (7) at (1, -1.25) {};
		\node [style=black dot] (8) at (-1.5, 1) {};
		\node [style=black dot] (9) at (-1.5, -1.25) {};
		\node [style=black dot] (10) at (3, 1) {};
		\node [style=none] (11) at (3, -2.5) {};
		\node [style=none] (12) at (-1.75, 1.25) {$x$};
		\node [style=none] (13) at (-1.75, -1.5) {$\exists y$};
		\node [style=none] (14) at (3.27, -2.75) {$z$};
		\node [style=none] (15) at (3.55, 1.25) {$f(x)$};
		\node [style=none] (16) at (0.75, 1.275) {$f$};
		\node [style=none] (17) at (0, -0.975) {$f$};
		\node [style=none] (18) at (-1.825, -0.25) {$Q_1$};
		\node [style=none] (19) at (3.325, -0.75) {$Q_2$};
		\node [style=none] (20) at (1.75, -0.75) {$f(y)$};
		\node [style=black dot] (21) at (1.5, -1.25) {};
		\node [style=black dot] (22) at (3, -2.5) {};
		\node [style=none] (23) at (2.025, -2.25) {$R_2$};
	\end{pgfonlayer}
	\begin{pgfonlayer}{edgelayer}
		\draw [style=mapsto] (4.center) to (5.center);
		\draw [style=dashed mapsto] (6.center) to (7.center);
		\draw [style=dashed Latex arrow] (9) to (8);
		\draw [style=Latex arrow] (11.center) to (10);
		\draw [style=dashed Latex arrow] (22) to (21);
	\end{pgfonlayer}
\end{tikzpicture}
\end{center}
\end{figure}

\noindent We will refer to this condition as being a \emph{weak p-morphism}\footnote{Observe that $f$ is a p-morphism with respect to $Q^{-1}$ provided $Q_2^{-1}[f(x)]=fQ_1^{-1}[x]$ for each $x\in X_1$. The above condition is weaker, thus justifying the name.} with respect to $Q^{-1}$.
\end{remark}

Clearly descriptive $\mipc$-frames and their morphisms form a category, which we denote by $\mipcfrm$.

\begin{remark}\label{rem:isos mipc}
Since the composition of morphisms in $\mipcfrm$ is the standard function composition, it follows that isomorphisms in $\mipcfrm$ are bijections that preserve and reflect the relations $R$ and $Q$ (see \cite[Prop.~1.4.15]{Esa19}). As a consequence, they also preserve and reflect the equivalence relation $E_Q$.
\end{remark}

We recall (see, e.g., \cite[Def.~3.38]{AHS06}) that two categories $\mathbf{A}$ and $\mathbf{B}$ are \emph{dually equivalent} provided $\mathbf{A}^{\text{op}}$ and $\mathbf{B}$ are equivalent. 

\begin{theorem}\cite[Thm.~17]{Bez99} \label{thm: duality for MHA}
$\mha$ is dually equivalent to $\mipcfrm$.
\end{theorem}

\begin{remark}\label{rem:ClopUp and Pf}
The functors establishing the above duality generalize the functors establishing Esakia duality \cite{Esa74,Esa19}.
If $\F=(X,R,Q)$ is a descriptive $\mipc$-frame, we let $\F^*$ be the 
Heyting algebra of clopen $R$-upsets of $X$ and define $\forall,\exists$ on $\F^*$ by
\[
\forall U = X \setminus Q^{-1}[X\setminus U] \ \mbox{ and } \ \exists U = Q[U].
\] 
If $f \colon \F_1 \to \F_2$ is a $\mipcfrm$-morphism, we define $f^* \colon \F_2^* \to \F_1^*$ by $f^*(U)=f^{-1}[U]$ for each $U \in \F_2^*$.

If $\A=(H, \forall, \exists)$ is a monadic Heyting algebra, we let $\A_*$ be the tuple $(X,R,Q)$ such that $(X,R)$ is the Esakia space\footnote{We recall that $X$ is the set of prime filters of $H$, $R$ is set-inclusion, and topology is given by the basis $\{\sigma(a) \setminus \sigma(b) : a,b \in H\}$, where $\sigma(a) = \{ x \in X : a \in x\}$.} of $H$ and $Q$ is defined by
\[
xQy \Longleftrightarrow x\cap H_0 \subseteq y.
\]
If $h \colon \A_1 \to \A_2$ is an $\mha$-morphism, we define $h_* \colon (\A_2)_* \to (\A_1)_*$ by $h_*(x)=h^{-1}[x]$ for each $x \in (\A_2)_*$.
\end{remark}

\begin{remark}\label{rem:onto and 1-1 mha mipcfrm}
Under the dual equivalence of \cref{thm: duality for MHA}, one-to-one morphisms in $\mha$ correspond to onto morphisms in $\mipcfrm$ and onto morphisms in $\mha$ to one-to-one morphisms in $\mipcfrm$ (see \cite[p.~39]{Bez99}).
\end{remark}

\section{$\ms$}\label{sec:ms4}

Let $\mathcal{L}_{\Box \forall}$ be a propositional modal language with two modalities $\Box$ and $\forall$. 

\begin{definition}\label{def:ms4}
The \emph{monadic $\sfour$}, denoted $\ms$, is the smallest set of formulas in 
$\mathcal{L}_{\Box \forall}$ containing all theorems of the classical propositional calculus $\cpc$, the $\sfour$-axioms for $\Box$, the $\sfive$-axioms for $\forall$, the left commutativity axiom
\[
\Box \forall p \to \forall \Box p,
\]
and closed under the rules of modus ponens, substitution, $\Box$-necessitation, and $\forall$-necessi\-tation.
\end{definition}

\begin{notation}
We let $\bbox$ denote the compound modality $\Box\forall$. 
\end{notation}

\begin{remark}
It is straightforward to check that $\bbox$ is an $\mathsf{S4}$-modality and that both $\bbox p \to \Box p$ and $\bbox p \to \forall p$ are provable in $\ms$. Therefore, $\bbox$ is a \emph{master modality} for $\ms$ (see, e.g., \cite[p.~71]{Kra99}).
\end{remark}

The algebraic semantics for $\ms$ is provided by $\ms$-algebras, called bimodal algebras in \cite[p.~145]{FS77}. 

\begin{definition}\label{def:ms-alg}
A \textit{monadic $\sfour$-algebra}, or an \emph{$\ms$-algebra} for short, is a tuple $\mathfrak B=(B, \Box, \forall)$ such that $B$ is a boolean algebra and $\Box, \forall$ are two unary functions on $B$ satisfying the axioms corresponding to the ones in \cref{def:ms4}.
\end{definition}

\begin{remark}\label{rem:equiv def ms4-alg}
\hfill\begin{enumerate}
\item A tuple 
$(B, \Box, \forall)$ is an $\ms$-algebra provided $(B, \Box)$ is an $\sfour$-algebra, $(B, \forall)$ is an $\sfive$-algebra, and $\Box \forall a \leq \forall \Box a$ for each $a \in B$.
\item As usual, we define $\exists \colon B \to B$ by $\exists a = \neg \forall \neg a$ for each $a \in B$.
\item Let $(B, \Box, \forall)$ be an $\ms$-algebra. We denote by $\bbox$ the unary function on $B$, obtained as the composition $\Box\forall$, corresponding to the master modality $\bbox$ of $\ms$. Then $(B, \bbox)$ is an $\sfour$-algebra and $\bbox a \leq \Box a, \forall a$ for each $a \in B$. As a consequence, the following identities hold for each $a \in B$:
\[
\bbox \Box a = \Box \bbox a = \bbox \forall a = \forall \bbox a = \bbox a.
\]
\end{enumerate}
 
\end{remark}

It immediately follows from \cref{def:ms-alg} that the class of $\ms$-algebras forms a variety. We denote the corresponding category by $\msa$.

Validity of formulas in $\ms$-algebras is defined as for monadic Heyting algebras, and so are extensions of $\ms$. We then have the following analogue of \cref{thm: alg comp MIPC,thm: lattice iso}:

\begin{theorem}\plabel{thm: lattice iso ms}
\hfill\begin{enumerate}
\item\label[thm: lattice iso ms]{thm: lattice iso ms:item1} $\ms\vdash\varphi$ iff $\msa \vDash \varphi$ for each formula $\varphi$ of $\mathcal{L}_{\Box \forall}$.
\item\label[thm: lattice iso ms]{thm: lattice iso ms:item2} The complete lattice $\Lambda(\ms)$ of extensions of $\ms$ is dually isomorphic to the complete lattice $\Lambda(\msa)$ of subvarieties of $\msa$.
\end{enumerate}
\end{theorem}

As in \cref{thm: lattice iso}, the above isomorphism is obtained by associating with each extension $\sf M$ of $\ms$ 
the variety $\Alg(\M)$ of $\ms$-algebras validating all formulas in $\M$. 
Conversely, each variety $\mathbb{V}$ of $\ms$-algebras gives rise to the extension $\Log(\mathbb{V})$ of $\ms$ consisting of the formulas valid in all members of $\mathbb{V}$. 

J\'onsson-Tarski duality for BAOs (boolean algebras with operators) takes on the following form for $\ms$-algebras.

\begin{definition}\plabel{def:descriptive ms-frame}
A \emph{descriptive $\ms$-frame} is a tuple $\mathfrak G=(Y,R,E)$ such that
\begin{enumerate}
\item\label[def:descriptive ms-frame]{def:descriptive ms-frame:item1} $Y$ is a Stone space,
\item\label[def:descriptive ms-frame]{def:descriptive ms-frame:item2} $R$ is a continuous quasi-order,
\item\label[def:descriptive ms-frame]{def:descriptive ms-frame:item3} $E$ is a continuous equivalence relation,
\item\label[def:descriptive ms-frame]{def:descriptive ms-frame:item4}
$x E y \; \& \; y R z \Longrightarrow \exists u \in Y : x R u \; \& \; u E z$.
\end{enumerate}
\begin{figure}[!ht]
\begin{center}
\begin{tikzpicture}
\node at (-0.2,-0.25) {$x$};
\node at (-0.25,2.27) {$\exists u$};
\node at (2.2,-0.25) {$y$};
\node at (2.2,2.25) {$z$};
\fill (0,0) circle(2pt);
\fill (0,2) circle(2pt);
\fill (2,0) circle(2pt);
\fill (2,2) circle(2pt);
\draw [dashed, -{Latex[width=1mm]}] (0,0) -- (0,2);
\draw [-{Latex[width=1mm]}] (2,0) -- (2,2);
\draw [dashed] (0,2) -- (2,2);
\draw (0,0) -- (2,0);
\node [above] at (1,2) {$E$};
\node [above] at (1,0) {$E$};
\node [left] at (0,1) {$R$};
\node [right] at (2,1) {$R$};
\end{tikzpicture}
\end{center}
\end{figure}
\end{definition}

\begin{definition}\plabel{def:msfrm-morphisms}
Let $\G_1=(Y_1,R_1,E_1)$ and $\G_2=(Y_2,R_2,E_2)$ be descriptive $\ms$-frames. A map $f \colon Y_1 \to Y_2$ is a \emph{morphism of descriptive $\ms$-frames} if 
\begin{enumerate}
\item\label[def:msfrm-morphisms]{def:msfrm-morphisms:item1} $f$ is continuous, 
\item\label[def:msfrm-morphisms]{def:msfrm-morphisms:item2} $R_2[f(y)]=fR_1[y]$ for each $y\in Y_1$,
\item\label[def:msfrm-morphisms]{def:msfrm-morphisms:item3} $E_2[f(y)]=fE_1[y]$ for each $y\in Y_1$.
\end{enumerate}
\end{definition}

\begin{remark}
\cref{def:msfrm-morphisms:item2} says that $f$ is a p-morphism with respect to $R$, and \cref{def:msfrm-morphisms:item3} that $f$ is a p-morphism with respect to $E$.
\end{remark}

Clearly descriptive $\ms$-frames and their morphisms form a category, which we denote by $\msfrm$.

\begin{remark}\label{rem:isos}
As in $\mipcfrm$, the composition of morphisms in $\msfrm$ is the standard function composition. Therefore, isomorphisms in $\msfrm$ are bijections that preserve and reflect the relations $R$ and $E$.
\end{remark}

We have the following version of J\'onsson-Tarski duality for $\ms$-algebras. For the reader's convenience, we give a sketch of proof (see also \cite[Rem.~2.14]{BM24}).

\begin{theorem}\label{thm: duality for ms4}
$\msa$ is dually equivalent to $\msfrm$. 
\end{theorem}

\begin{proof}[Sketch of proof.]
The functors establishing the dual equivalence are described as follows.
The functor $(-)^* \colon \msfrm \to \msa$ sends
a descriptive $\ms$-frame $\G=(Y,R,E)$ to $\G^*=(\Clop(Y),\Box,\forall)$, where $\Clop(Y)$ is the boolean algebra of clopens of $Y$ and
$\Box,\forall$ are defined on $\Clop(Y)$ by 
\[
\Box U = Y \setminus R^{-1}[Y \setminus U] \ \mbox{ and } \ \forall U = Y \setminus E[Y \setminus U].
\] 
Since $R$ is reflexive and transitive, $\Box$ is an $\mathsf{S4}$-operator; since $E$ is an equivalence relation, $\forall$ is an $\mathsf{S5}$-operator; and \cref{def:descriptive ms-frame:item4} yields that $\Box \forall U \leq \forall \Box U$. Therefore, $\G^* \in \msa$. If $f$ is a $\msfrm$-morphism, then $f^*$ is $f^{-1}$. Clearly, $f^{-1}$ is a boolean homomorphism and \crefrange{def:msfrm-morphisms:item2}{def:msfrm-morphisms:item3} yields that $f^{-1}$ is an $\msa$-morphism. Thus, $(-)^*$ is well defined.
 
The functor $(-)_* \colon \msa \to \msfrm$ sends an $\ms$-algebra $\B=(B, \Box, \forall)$
to the tuple $\B_*=(Y,R,E)$ such that
$(Y,R)$ is the dual of the $\sfour$-algebra $(B,\Box)$\footnote{That is, $Y$ is the Stone space (the space of ultrafilters) of $B$ and for any $x,y \in Y$ we have $xRy$ iff $\Box^{-1}x \subseteq y$.} and
\[
xEy \Longleftrightarrow x \cap B_0 = y \cap B_0,
\]
where 
$B_0 = \{\forall a : a \in B\}$.
Then $Y$ is a Stone space, $R$ is a continuous quasi-order, and $E$ is a continuous equivalence relation.
The commutativity axiom $\Box \forall a \le \forall \Box a$ for every $a \in B$ 
implies that
$ER^{-1}[U] \subseteq R^{-1}E[U]$ for every $U \in \Clop(Y)$. A standard argument using Esakia's Lemma (see \cite[Lem.~3.3.12]{Esa19} or \cite[p.~350]{CZ97}) then yields that $ER^{-1}[y] \subseteq R^{-1}E[y]$ for every $y \in Y$. It is straightforward to see that this last condition is equivalent to  \cref{def:descriptive ms-frame:item4}. Therefore, $\B_* \in \msfrm$. If $f$ is an $\msa$-morphism, then $f_*$ is $f^{-1}$. 
It follows from J\'onsson-Tarski duality that $f^{-1}$ is an $\msa$-morphism, so $(-)_*$ is well defined, and that $(-)^*$ and $(-)_*$ establish the desired dual equivalence between $\msa$ and $\msfrm$. 
\end{proof}

\begin{notation}
We use the same notation for the functors establishing duality between $\msa$ and $\msfrm$ and the ones establishing duality between $\mha$ and $\mipcfrm$, but it will always be clear from the context which functor we work with. To avoid confusion, we will always denote monadic Heyting algebras by $\A$ and $\ms$-algebras by $\B$. Similarly, we will denote descriptive $\mipc$-frames by $\F$ and descriptive $\ms$-frames by $\G$.
\end{notation}

\section{The G\"odel translation} \label{sec: Godel translation}

We recall (see, e.g., \cite[p.~96]{CZ97}) that the G\"odel translation $(-)^t$ of $\ipc$ into $\sf S4$ is defined by
\begin{align*}
& \bot^t = \bot\\
& p^t = \Box p \quad \mbox{for each propositional letter } p \\
& (\varphi \land \psi)^t = \varphi^t \land \psi^t \\
& (\varphi \lor \psi)^t = \varphi^t \lor \psi^t \\
& (\varphi \to \psi)^t = \Box (\neg \varphi^t \lor \psi^t).
\end{align*}

Fischer Servi \cite{FS77} (see also \cite{FS78a}) extended the G\"odel translation to a translation of $\mipc$ into $\ms$ as follows:
\begin{align*}
(\forall \varphi)^t &= \bbox \varphi^t \\
(\exists \varphi)^t &= \exists \varphi^t. 
\end{align*}

\begin{theorem}\cite{FS77}\label{thm: MS4 modal comp of MIPC}
$\mipc \vdash \varphi$ iff $\ms \vdash \varphi^t$ for each $\varphi$ in $\Lae$.
\end{theorem}

There are three well-known maps between the lattices of extensions of $\ipc$ and $\sf S4$ (see, e.g., \cite[Sec.~9.6]{CZ97}) that have obvious generalizations to the monadic setting.

\begin{definition}\
\begin{enumerate}[label=\normalfont(\arabic*)]
\item Define $\rho \colon \Lambda(\ms) \to \Lambda(\mipc)$ by $\rho \M = \{ \varphi : \M \vdash \varphi^t \}$.
\item Define $\tau \colon \Lambda(\mipc) \to \Lambda(\ms)$ by $\tau \L = \ms + \{\varphi^t : \L \vdash \varphi \}$.
\item Define $\sigma \colon \Lambda(\mipc) \to \Lambda(\mgrz)$ by $\sigma \L = \mgrz + \{\varphi^t : \L \vdash \varphi \}$.
\end{enumerate}
\end{definition}

Here $\mgrz \coloneqq \ms +\grz$ is the \emph{monadic Grzegorczyk logic} introduced by Esakia \cite{Esa88}, and $\Lambda(\mgrz)$ is the complete lattice of extensions of $\mgrz$.

\begin{proposition}\plabel{prop:rho tau sigma adjoints}
\hfill\begin{enumerate}
\item\label[prop:rho tau sigma adjoints]{prop:rho tau sigma adjoints:item1} $\tau$ is left adjoint to $\rho$.
\item\label[prop:rho tau sigma adjoints]{prop:rho tau sigma adjoints:item2} $\sigma$ is left adjoint to the restriction of $\rho$ to $\Lambda(\mgrz)$.
\item\label[prop:rho tau sigma adjoints]{prop:rho tau sigma adjoints:item3} $\tau$ and $\sigma$ preserve arbitrary joins, and $\rho$ preserves arbitrary meets.
\end{enumerate}
\end{proposition}

\begin{proof}
\eqref{prop:rho tau sigma adjoints:item1}. 
For all $\L \in \Lambda(\mipc)$ and $\M \in \Lambda(\ms)$, 
\[
\tau \L \subseteq \M \iff \{\varphi^t : \L \vdash \varphi \} \subseteq \M \iff \mathsf{L}\subseteq \{\varphi : \mathsf{M}
\vdash \varphi ^{t} \} \iff \L \subseteq \rho \M.
\]
Thus, $\tau$ is left adjoint to $\rho$.

\eqref{prop:rho tau sigma adjoints:item2}. If 
$\M \in \Lambda(\mgrz)$, then \eqref{prop:rho tau sigma adjoints:item1} 
yields $\sigma \L \subseteq \M$ iff $\L \subseteq \rho \M$. Thus, $\sigma$ is left adjoint to the restriction of $\rho$ to $\Lambda(\mgrz)$.

\eqref{prop:rho tau sigma adjoints:item3}. This is obvious since 
left adjoints preserve joins and right adjoints preserve meets (see, e.g., \cite[Prop.~7.34]{DP02}).
\end{proof}

The notions of a modal companion and the intuitionistic fragment have obvious generalizations 
to the monadic setting:

\begin{definition}
Let $\L\in\Lambda(\mipc)$ 
and $\M\in\Lambda(\msfour)$. 
If $\L= \rho \M$, then we call $\M$ a \emph{modal companion} of $\L$ and $\L$ the \emph{intuitionistic fragment} of $\M$. 
\end{definition}

\begin{theorem}\
\begin{enumerate}[label=\normalfont(\arabic*)]
\item\cite{FS77} $\msfour$ is a modal companion of $\mipc$.
\item\cite{Esa88} $\mgrz$ is a modal companion of $\mipc$.
\end{enumerate}
\end{theorem}

In \cref{sec: failure BE} we will show that $\tau$ and $\sigma$ are lattice homomorphisms, thus generalizing the corresponding results of Maksimova and Rybakov \cite{MR74eng} (see also \cite[Thm.~9.66]{CZ97}). 
On the other hand, we will prove that $\rho$ is neither a lattice homomorphism nor one-to-one. This yields that $\sigma$ is not an isomorphism, and hence the Blok--Esakia Theorem fails in the monadic setting.
For this we need a semantic characterization of $\rho$, which will be done 
in the next section.

\section{The functor $\Op$}\label{sec:Op}

Let $(B,\Box)$ be an $\sfour$-algebra. We recall that $a\in B$ is 
\emph{open} if $\Box a = a$. It is well known (see, e.g., \cite[Prop.~2.2.4]{Esa19}) that the set $H$ of open elements of $(B,\Box)$ is a bounded sublattice of $B$ which forms 
a Heyting algebra, where $a \to_H b = \Box(\neg a \vee b)$. Moreover, this correspondence extends to a functor $\Op \colon \sa \to \ha$ (see, e.g., \cite[Thm.~2.2.5]{Esa19}). We extend $\Op$ to the monadic setting.

\begin{definition}\label{def:Op(B)}
For an $\ms$-algebra $\mathfrak B = (B, \Box, \forall)$ let $\Op(\mathfrak B) = (H, \forall_H, \exists_H)$, where $H$ is the Heyting algebra of open elements of $(B,\Box)$, $\forall_H a = \bbox a$, and $\exists_H a = \exists a$ for all $a\in H$.
\end{definition}

\begin{theorem}\cite[Thm.~4]{FS77}
If $\mathfrak B$ is an $\ms$-algebra, then $\Op(\mathfrak{B})$ is a monadic Heyting algebra.
\end{theorem}

For each $\ms$-morphism $h \colon \mathfrak{B}_1 \to \mathfrak{B}_2$ let $\Op(h) \colon \Op(\mathfrak{B}_1) \to \Op(\mathfrak{B}_2)$ be its restriction. It is straightforward to see that $\Op(h)$ is an $\mha$-morphism and that \cite[Thm.~2.2.5]{Esa19} 
generalizes to the following: 

\begin{theorem} 
$\Op \colon \msa \to \mha$ is a functor.
\end{theorem}

\begin{theorem}\cite[Thm.~5]{FS77}\label{thm:Op and translation}
Let $\varphi$ be a formula in the language $\Lae$ and $\mathfrak{B}$ an $\ms$-algebra. Then $\Op(\mathfrak{B}) \vDash \varphi$ iff $\mathfrak{B} \vDash \varphi^t$.
\end{theorem}

In order to see how $\rho$ relates to $\Op$, 
we need to describe the behavior of  
$\Op$ with respect 
to the class operators $\H$, $\S$, and $\P$ of taking homomorphic images, subalgebras, and products.

Let $(B,\Box)$ be an $\sfour$-algebra and 
$H$ the Heyting algebra of its open elements. 
We recall that a filter $G$ of $B$ is a {\em $\Box$-filter} if $a\in G$ implies $\Box a\in G$. There is a well-known isomorphism
between the congruences of $(B,\Box)$, the $\Box$-filters of $(B,\Box)$, the filters of $H$, and the congruences of $H$ (see, e.g., \cite[Sec.~2.4]{Esa19}). This readily generalizes to the monadic setting. 

\begin{definition}
A filter $F$ of a monadic Heyting algebra $\mathfrak{A} = (H,\forall,\exists)$ is \emph{monadic} if $a\in F$ implies $\forall a\in F$; and a $\Box$-filter $G$ of an $\ms$-algebra $\mathfrak{B}=(B, \Box, \forall)$ is {\em monadic} if it satisfies the same condition (equivalently, $G$ is a monadic filter provided $a \in G$ implies $\bbox a \in G$). 
\end{definition}

\begin{theorem}\label{thm:correspondence congruences filters}
Let $\mathfrak B$ be an $\ms$-algebra and $\mathfrak A=\Op(\mathfrak B)$. There is an isomorphism between the posets of congruences of $\mathfrak B$, monadic filters of $\mathfrak B$, monadic filters of $\mathfrak A$, and congruences of $\mathfrak A$. 
\end{theorem}

\begin{proof}[Sketch of proof.]
That the poset of congruences of $\mathfrak A$ is isomorphic to the poset of monadic filters of $\mathfrak A$ follows from \cite[Thm.~2.7]{Bez98b}. A similar argument yields that the poset of congruences of $\mathfrak B$ is isomorphic to the poset of monadic filters of $\mathfrak B$ (see, e.g., \cite[Thm.~3.4]{BM24}).
Finally, the isomorphism between the posets of monadic filters of $\mathfrak B$ and $\mathfrak A$ directly generalizes the well-known isomorphism between the $\Box$-filters of $\mathfrak B$ and the filters of $\mathfrak A$ obtained by 
\[
G \mapsto G \cap \mathfrak A \ \mbox{ and } \ F \mapsto {\uparrow}_{\mathfrak B} F,
\] 
where
${\uparrow}_{\mathfrak B} F = \{b \in B : b \ge a \text{ for some } a \in F\}$.
\end{proof}

As an immediate corollary, we obtain the following correspondence between subdirectly irreducible algebras (for all notions of universal algebra we refer to \cite{BS81}).

\begin{corollary}\label{cor:B sub irr iff O(B) sub irr}
$\B$ is a subdirectly irreducible $\ms$-algebra iff $\Op(\B)$ is a subdirectly irreducible monadic Heyting algebra.
\end{corollary}

Consequently, both 
$\msa$ and $\mha$ are congruence-distributive 
and have the congruence extension property (CEP). This will be used in what follows, and so will be the next lemma. 

\begin{lemma}\plabel{lem:quotients and op}
Let $\B$ be an $\ms$-algebra.
\begin{enumerate}
\item\label[lem:quotients and op]{lem:quotients and op:item1} If $G$ is a monadic filter of $\mathfrak{B}$, then $\Op(\mathfrak{B}/G) \cong \Op(\mathfrak{B})/(G \cap \Op(\mathfrak{B}))$.
\item\label[lem:quotients and op]{lem:quotients and op:item2} If $F$ is a monadic filter of $\Op(\mathfrak{B})$, then $\Op(\mathfrak{B})/F \cong \Op(\mathfrak{B}/\up_{\mathfrak B} F)$.
\end{enumerate}
\end{lemma}

\begin{proof}
\eqref{lem:quotients and op:item1}.
Let $\pi \colon \mathfrak{B} \to \mathfrak{B}/G$ be the quotient map.
Then $\Op(\pi) \colon \Op(\mathfrak{B}) \to \Op(\mathfrak{B}/G)$ is an onto $\mha$-morphism. 
Moreover, for $b \in \Op(\mathfrak{B})$, we have 
\[
\Op(\pi)(b)=1 \iff \pi(b)=1 \iff b \in G.
\] 
Thus, $G \cap \Op(\mathfrak{B})$ is the kernel of $\Op(\pi)$, and hence the First Isomorphism Theorem \cite[Thm.~II.6.12]{BS81} 
implies that $\Op(\mathfrak{B}/G) \cong \Op(\mathfrak{B})/(G \cap \Op(\mathfrak{B}))$.

\eqref{lem:quotients and op:item2}.
Since $F=\up_{\mathfrak B} F \cap \Op(\mathfrak{B})$, 
\eqref{lem:quotients and op:item1} implies that 
\[
\Op(\mathfrak{B})/F = \Op(\mathfrak{B})/(\up_{\mathfrak B} F \cap \Op(\mathfrak{B})) \cong \Op(\mathfrak{B}/\up_{\mathfrak B} F).\qedhere
\]
\end{proof}

For a class of algebras $\K$,
we denote by $\H(\K)$, $\S(\K)$, and $\P(\K)$, the classes of homomorphic images, subalgebras, and products of algebras from $\K$.
Let $\V(\K)$ be the variety generated by $\K$.
It is well known (see, e.g., \cite[Thm.~II.9.5]{BS81}) that $\V(\K) = \H\S\P(\K)$. When $\K$ is a class of $\ms$-algebras, we write $\Op(\K)$ for $\{ \Op(\mathfrak{B}) : \mathfrak{B} \in \K \}$.

\begin{proposition}\label{prop:HSP and Op}
For a class $\K$ of $\ms$-algebras, we have 
$\Op\H(\K)=\H\Op(\K)$, $\Op\P(\K)=\P\Op(\K)$, and 
$\Op\S(\K) \subseteq \S\Op(\K)$.
\end{proposition}

\begin{proof}
First, let $\mathfrak{A} \in \Op\H(\K)$, so $\A =\Op(\mathfrak{B}')$ for a homomorphic image $\mathfrak{B}'$ of some $\mathfrak{B} \in \K$. Then $\mathfrak{B}' \cong \mathfrak{B}/G$ for some monadic $\Box$-filter $G$ of $\mathfrak{B}$. By \cref{lem:quotients and op:item1}, 
\[
\Op(\mathfrak{B}') \cong \Op(\mathfrak{B}/G) \cong \Op(\mathfrak{B})/(G \cap \Op(\mathfrak{B})),
\]
and so $\Op(\mathfrak{B}')$ is a homomorphic image of $\Op(\mathfrak{B})$. Thus, $\A \in \H\Op(\K)$, and so $\Op\H(\K) \subseteq \H\Op(\K)$. For the other inclusion, let $\mathfrak{A} \in \H\Op(\K)$, so $\A$ is a homomorphic image of $\Op(\mathfrak{B})$ with $\mathfrak{B} \in \K$. Then $\mathfrak{A} \cong \Op(\mathfrak{B})/F$ for some monadic filter $F$ of $\Op(\mathfrak{B})$. By \cref{lem:quotients and op:item2}, 
\[
\mathfrak{A} \cong \Op(\mathfrak{B})/F \cong \Op(\mathfrak{B}/\up_{\mathfrak B} F),
\]
and hence $\mathfrak{A} = \Op(\mathfrak{B}')$ for some $\B' \cong \mathfrak{B}/\up_{\mathfrak B} F$. 
Thus, $\A \in \Op\H(\K)$, and so $\H\Op(\K) \subseteq \Op\H(\K)$.

Next, let $\{ \mathfrak{B}_i : i \in I \}$ be 
a family of $\ms$-algebras. By \cite[Thm.~2.2.5]{Esa19}, $\Op(\Pi_{i \in I}\mathfrak{B}_i)=\Pi_{i \in I} \Op(\mathfrak{B}_i)$ as Heyting algebras. 
Moreover, since $\bbox$ and $\exists$ are componentwise, the two algebras are equal as monadic Heyting algebras. Thus, $\Op\P(\K)=\P\Op(\K)$.

Finally, let $\mathfrak{A} \in \Op\S(\K)$, so $\A = \Op(\B')$ for a subalgebra $\mathfrak{B}'$ of some $\mathfrak{B} \in \K$. Then $\Op(\mathfrak{B}')$ is a Heyting subalgebra of $\Op(\mathfrak{B})$ (see, e.g., \cite[Thm.~2.2.5]{Esa19}). Moreover,
$\bbox a, \exists a \in \mathfrak{B}'$ for each $a \in \Op(\mathfrak{B}')$. Therefore, $\Op(\mathfrak{B}')$ is also a monadic subalgebra of $\Op(\mathfrak{B})$. Thus, $\A \in \S\Op(\K)$, and so $\Op\S(\K)\subseteq\S\Op(\K)$.
\end{proof}

\begin{remark} \label{rem: SO neq OS}
In \cref{thm:SO not well behaved:item2} we will see that the inclusion $\S\Op(\K) \subseteq \Op\S(\K)$ does not hold in general.
This is in contrast with the functor $\Op \colon \ha \to \sa$ which is well known to commute with $\H$, $\S$, and $\P$, and hence maps varieties of $\sfour$-algebras to varieties of Heyting algebras (see, e.g., \cite[Cor.~2.2.6]{Esa19}).
\end{remark}

\begin{proposition}\plabel{prop:SOpV=VOp}
\hfill\begin{enumerate}
\item\label[prop:SOpV=VOp]{prop:SOpV=VOp:item1} If $\K$ is a class of $\ms$-algebras, then $\V(\Op(\K)) = \S\Op(\V(\K))$.
\item\label[prop:SOpV=VOp]{prop:SOpV=VOp:item2} If $\mathbb{V}$ is a variety of $\ms$-algebras, then $\S\Op(\mathbb{V})$ is the variety generated by $\Op(\mathbb{V})$. 
\end{enumerate}
\end{proposition}

\begin{proof}
\eqref{prop:SOpV=VOp:item1}. By \cref{prop:HSP and Op}, $\H$ and $\Op$ commute. Since $\mha$ has the CEP, $\H$ and $\S$ commute on subclasses of $\mha$ by \cite[p.~62]{BS81}. Therefore,
\begin{align*}
\S\Op(\V(\K)) = \S\Op(\H\S\P(\K)) = \S\H\Op(\S\P(\K)) = \H\S\Op(\S\P(\K)).
\end{align*}
If $\K'$ is a class of $\ms$-algebras, then \cref{prop:HSP and Op} implies that $\S\Op(\S(\K')) \subseteq \S\Op(\K')$, and hence $\S\Op(\S(\K')) = \S\Op(\K')$. Thus,
\begin{align*}
\H\S\Op(\S\P(\K)) = \H\S\Op(\P(\K)) = \H\S\P\Op(\K) = \V(\Op(\K))
\end{align*}
because $\P$ and $\Op$ commute by \cref{prop:HSP and Op}.
Consequently, $\S\Op(\V(\K)) = \V(\Op(\K))$.

\eqref{prop:SOpV=VOp:item2}. 
By \eqref{prop:SOpV=VOp:item1}, $\S\Op(\mathbb{V}) = \S\Op(\V(\mathbb{V})) =  \V(\Op(\mathbb{V}))$.
\end{proof}

\begin{theorem}\plabel{thm:int frag SrhoV}
\hfill\begin{enumerate}
\item\label[thm:int frag SrhoV]{thm:int frag SrhoV:item1} Let $\M\in\Lambda(\ms)$. Then $\Alg(\rho \M) = \S\Op(\Alg(\M))$.
\item\label[thm:int frag SrhoV]{thm:int frag SrhoV:item2} Let $\L\in\Lambda(\mipc)$. Then $\M\in\Lambda(\ms)$ is a modal companion of $\L$ iff $\Alg(\L) = \S\Op(\Alg(\M))$.
\item\label[thm:int frag SrhoV]{thm:int frag SrhoV:item3} $\S\Op$ commutes with arbitrary joins of varieties of $\ms$-algebras.
\end{enumerate}
\end{theorem}

\begin{proof}
\eqref{thm:int frag SrhoV:item1}. By algebraic completeness and \cref{thm:Op and translation}, for each formula $\varphi$ in $\Lae$, we have 
$\M \vdash \varphi^t$ iff $\Alg(\M) \vDash \varphi^t$ iff $\Op(\Alg(\M)) \vDash \varphi$. 
Therefore, $\rho\M\vdash\varphi$ iff $\Op(\Alg(\M)) \vDash \varphi$.
Thus, by \cref{prop:SOpV=VOp:item1},
\[
\Alg(\rho\M)
=\V(\Op(\Alg(\M)))=\S\Op(\Alg(\M)).
\] 
 
\eqref{thm:int frag SrhoV:item2}.
$\M$ is a modal companion of $\L$ iff $\L=\rho\M$, which is equivalent to $\Alg(\L)=\Alg(\rho \M)$. By \eqref{thm:int frag SrhoV:item1}, this is equivalent to $\Alg(\L) = \S\Op(\Alg(\M))$.

\eqref{thm:int frag SrhoV:item3}. Let $\{ \mathbb{V}_i : i \in I\}$ be a family of varieties of $\ms$-algebras. Since meets in $\Lambda(\mipc)$ and $\Lambda(\ms)$ are intersections, by \eqref{thm:int frag SrhoV:item1} and \cref{prop:rho tau sigma adjoints:item3} we have
\begin{align*}
\S \Op \left(\bigvee \{ \mathbb{V}_i : i \in I\} \right) & = \S \Op \left(\bigvee \{ \Alg(\Log(\mathbb{V}_i)) : i \in I\} \right) = \S \Op \Alg\left(\bigcap \{ \Log(\mathbb{V}_i) : i \in I\} \right) \\
& = \Alg\left(\rho \bigcap \{ \Log(\mathbb{V}_i) : i \in I\} \right) = \Alg\left(\bigcap \{ \rho\Log(\mathbb{V}_i) : i \in I\} \right)\\
& = \bigvee \{ \Alg(\rho\Log(\mathbb{V}_i)) : i \in I\} = \bigvee \{ \S\Op(\Alg(\Log(\mathbb{V}_i))) : i \in I\}\\
& = \bigvee \{ \S\Op(\mathbb{V}_i) : i \in I\}.\qedhere
\end{align*}
\end{proof}

We conclude this section by describing a functor from the category of descriptive $\ms$-frames to the category of descriptive $\mipc$-frames that is dual 
to $\Op$. For this we generalize the notion of the {\em skeleton} of an $\sf S4$-frame (see, e.g., \cite[p.~68]{CZ97}) to the monadic setting.

\begin{definition}\label{def:skeleton}
For a descriptive $\ms$-frame $\G=(Y,R,E)$, define $\sk(\G)=(X,R',Q')$ as follows. Let $X\coloneqq Y/E_R$ be the quotient of $Y$ by the equivalence relation $E_R$ on $Y$ induced by $R$, and let $\pi \colon Y \to X$ be the quotient map. Define $R'$ on $X$ by 
\[
\pi(x) R' \pi(y) \iff x R y.
\] 
Also, let $Q$ be the composite $E \circ R$, and define $Q'$ on $X$ by 
\[
\pi(x) Q' \pi(y) \iff x Q y.
\]
\end{definition}

\begin{theorem}\label{thm:sk(G) is descriptive}
$\sk(\G)$ is a descriptive $\mipc$-frame.
\end{theorem}

\begin{proof}
By \cite[Lem.~3.4.13]{Esa19}, $X$ is a Stone space and $R'$ is a continuous relation on $X$. In addition, the quotient map $\pi \colon Y \to X$ is continuous and satisfies $R'[\pi(x)]=\pi[R[x]]$ for each $x \in Y$. By \cite[Lem.~3.3]{BBI23}, $Q'$ is a well-defined quasi-order satisfying \crefrange{def:ono:item5}{def:ono:item6}.
We show that $Q'$ is a continuous relation. Since both $R$ and $E$ are continuous relations on $Y$, so is $Q$. 
For $A \subseteq X$, it is straightforward to see that 
\[
\pi^{-1}[Q'[A]]=Q[\pi^{-1}[A]] \mbox{ and } \pi^{-1}[(Q')^{-1}[A]]=Q^{-1}[\pi^{-1}[A]].
\] 
Let $x \in Y$. Since 
\[
\pi^{-1}[Q'[\pi(x)]]=Q[\pi^{-1}[\pi(x)]]=Q[E_R[x]] = Q[x]
\]
and $Q$ is continuous, it follows from the definition of the quotient topology that
$Q'[\pi(x)]$ is closed in $X$. Let $U \subseteq X$ be clopen. Then 
\[
\pi^{-1}[(Q')^{-1}[U]]=Q^{-1}[\pi^{-1}[U]],
\]
which is clopen in $Y$ because $Q$ is continuous and $\pi^{-1}[U]$ is clopen in $Y$. Thus, $(Q')^{-1}[U]$ is clopen in $X$, and hence $Q'$ is continuous.

It is left to verify \cref{def:ono:item4}. Let $U$ be a clopen $R'$-upset of $X$. Then $\pi^{-1}[U]$ is a clopen $R$-upset of $Y$. 
Since $E$ is a continuous equivalence relation,
\[
\pi^{-1}[Q'[U]]=Q[\pi^{-1}[U]]=ER[\pi^{-1}[U]]=E[\pi^{-1}[U]] = E^{-1}[\pi^{-1}[U]]
\]
is clopen. Thus, $Q'[U]$ is clopen, and hence $\sk(\G)$ is a descriptive $\mipc$-frame. 
\end{proof}

As with descriptive $\sf S4$-frames, if $f \colon \G_1 \to \G_2$ is a $\msfrm$-morphism, we define the map $\sk(f) \colon \sk(\G_1) \to \sk(\G_2)$ by 
\[
\sk(f)(\pi_1(x))=\pi_2(f(x))
\]
for each $x \in \G_1$, where $\pi_1,\pi_2$ are the corresponding quotient maps.

\begin{lemma}
$\sk \colon \msfrm \to \mipcfrm$ is a well-defined functor.
\end{lemma}

\begin{proof}
By~\cref{thm:sk(G) is descriptive}, $\sk$ is well defined on objects.
Let $f \colon \G_1 \to \G_2$ be a $\msfrm$-morphism with $\G_1=(Y_1,R_1,E_1)$ and $\G_2=(Y_2,R_2,E_2)$. It is well known that $\sk(f)$ is a well-defined continuous p-morphism with respect to $R'$. For the reader's convenience, we sketch a proof. Since $f[R_1[z]] \subseteq R_2[f(z)]$ for each $z \in Y_1$, we have that $xE_{R_1}y$ implies $f(x)E_{R_2}f(y)$ for each $x,y \in Y_1$. Thus, $\sk(f)$ is a well-defined function. It is continuous because $\pi_2 \circ f = \sk(f) \circ \pi_1$ and $\pi_2 \circ f$ is continuous.
For $x \in Y_1$ we have 
\begin{align*}
\sk(f)[R_1'[\pi_1(x)]] &= \sk(f)[\pi_1[R_1[x]]]=\pi_2[f[R_1[x]]] = \pi_2[R_2[f(x)]]\\
& = R_2'[\pi_2[f(x)]]=R_2'[\sk(f)(\pi_1(x))],
\end{align*}
where the first and fourth equalities follow from the definitions of $R_1'$ and $R_2'$, the second and last equalities from the definition of $\sk(f)$, and the third equality holds because $f$ is a p-morphism with respect to $R$. Thus, $\sk(f)$ is a p-morphism with respect to $R'$. 

Since $Q_1=E_1 \circ R_1$ and $f$ is a p-morphism with respect to $R$ and $E$, it is also a p-morphism with respect to $Q$. Then a similar chain of equalities yields that $\sk(f)$ is a p-morphism with respect to $Q'$. 

We show that $\sk(f)$ is a weak p-morphism with respect to $(Q')^{-1}$. Let $x \in Y_1$. Then
\begin{equation}\label{eq:dagger}
Q_2^{-1}[f(x)]=R_2^{-1} E_2[f(x)]=R_2^{-1} f[E_1[x]] = R_2^{-1}[f[R_1^{-1}E_1[x]]] = R_2^{-1} f[Q_1^{-1}[x]],
\tag{$\dagger$}
\end{equation}
where the first and last equalities follow from the definitions of $Q_2$ and $Q_1$, the second equality holds because $f$ is a p-morphism with respect to $E$, the left to right inclusion in the third equality is a consequence of the reflexivity of $R_1$, and the right to left inclusion holds because $f$ preserves $R_1$ and $R_2$ is transitive.
Consequently,
\begin{align*}
(Q_2')^{-1}[\sk(f)(\pi_1[x])] &= (Q_2')^{-1}[\pi_2(f(x))] = \pi_2 [Q_2^{-1}[f(x)]]\\
& = \pi_2 [R_2^{-1}[f[Q_1^{-1}[x]]]] = (R_2')^{-1}[\pi_2[f[Q_1^{-1}[x]]]]\\
& = (R_2')^{-1}[\sk(f)[\pi_1[Q_1^{-1}[x]]]] =(R_2')^{-1}[\sk(f)[(Q_1')^{-1}[\pi_1[x]]]],
\end{align*}
where the first and fifth equalities follow from the definition of $\sk(f)$, the second, fourth, and last equalities are consequences of the definitions of $Q_2'$, $R_2'$, and $Q_1'$, and the third equality follows from \eqref{eq:dagger}. Thus, $\sk(f)$ is an $\mipc$-morphism (see \cref{def:mipcfrm-morphisms}).
That $\rho$ preserves compositions and identities is an immediate consequence of its definition. Therefore, $\sk \colon \msfrm \to \mipcfrm$ is a well-defined functor.
\end{proof}

\begin{example}\label{ex:sk(f) not p-morph wrt E}
It is not true in general that if $f$ is a $\msfrm$-morphism, then $\sk(f)$ is a p\nobreakdash-morphism with respect to $E_{Q'}$. To see this, consider $\mathfrak{H}_1=(Y_1,R_1,E_1)$ and $\mathfrak{H}_2=(Y_2,R_2,E_2)$ depicted in \cref{fig:sk(f) not p-morph wrt E}(a). 
The black arrows 
represent the quasi-orders $R_i$, the double black arrows the $E_{R_i}$-equivalence classes, and the red circles 
the $E_i$-equivalence classes for $i=1,2$.\footnote{These should be understood as Hasse diagrams of the corresponding frames. For example, $aR_1b$ and $bR_1d$, so $aR_1d$, but the arrow from $a$ to $d$ is not drawn.} 
\begin{figure}[ht]
\begin{tikzpicture}[-{Latex[width=1mm]}]
\coordinate (1BL) at (-5,0);
\coordinate (1BR) at (-3,0);
\coordinate (1TL) at (-5,2);
\coordinate (1TR) at (-3,2);
\fill (1BL) circle(2pt);
\fill (1BR) circle(2pt);
\fill (1TL) circle(2pt);
\fill (1TR) circle(2pt);
\draw (1BL) -- (1TL);
\draw (1BR) -- (1TR);
\draw[{Latex[width=1mm]}-{Latex[width=1mm]}] (1TL) -- (1TR);
\clustertwo{1BL}{1TL}{1.35}{1.2};
\clusterone{1BR}{1.2};
\clusterone{1TR}{1.2};
\node at ([shift={(0:0.6)}]1BL) {$a$};
\node at ([shift={(0:0.6)}]1BR) {$c$};
\node at ([shift={(30:0.6)}]1TL) {$b$};
\node at ([shift={(30:0.6)}]1TR) {$d$};
\node at (-4,-1) {$\mathfrak{H}_1$};
\coordinate (2BL) at (2,0);
\coordinate (2TR) at (4,2);
\coordinate (2TL) at (2,2);
\fill (2BL) circle(2pt);
\fill (2TR) circle(2pt);
\fill (2TL) circle(2pt);
\draw (2BL) -- (2TL);
\draw[{Latex[width=1mm]}-{Latex[width=1mm]}] (2TL) -- (2TR);
\clustertwo{2BL}{2TL}{1.35}{1.2};
\clusterone{2TR}{1.2};
\node at ([shift={(0:0.6)}]2BL) {$u$};
\node at ([shift={(30:0.6)}]2TL) {$v$};
\node at ([shift={(30:0.6)}]2TR) {$w$};
\node at (3,-1) {$\mathfrak{H}_2$};
\node at (-0.5,-1.5) {(a)};
\coordinate (3BL) at (-5,-5);
\coordinate (3BR) at (-3,-5);
\coordinate (3T) at (-4,-3);
\fill (3BL) circle(2pt);
\fill (3BR) circle(2pt);
\fill (3T) circle(2pt);
\draw (3BL) -- (3T);
\draw (3BR) -- (3T);
\clustertwo{3BL}{3T}{1.3}{1.25};
\clusterone{3BR}{1.2};
\node at ([shift={(-90:0.6)}]3BL) {$\pi_1(a)$};
\node at ([shift={(90:0.5)}]3T) {$\pi_1(b)=\pi_1(d)$};
\node at ([shift={(-90:0.6)}]3BR) {$\pi_1(c)$};
\node at (-4,-6.5) {$\sk(\mathfrak{H}_1)$};
\coordinate (4B) at (3,-5);
\coordinate (4T) at (3,-3);
\fill (4B) circle(2pt);
\fill (4T) circle(2pt);
\draw (4B) -- (4T);
\clustertwo{4B}{4T}{1.35}{1.2};
\node at ([shift={(-90:0.6)}]4B) {$\pi_2(u)$};
\node at ([shift={(90:0.6)}]4T) {$\pi_2(v)=\pi_2(w)$};
\node at (3,-6.5) {$\sk(\mathfrak{H}_2)$};
\node at (-0.5,-7) {(b)};
\end{tikzpicture}
\caption{The counterexample from \cref{ex:sk(f) not p-morph wrt E}.}\label{fig:sk(f) not p-morph wrt E}
\end{figure}

We show that $\mathfrak{H}_1$ is a descriptive $\ms$-frame. Clearly $R_1$ is a quasi-order and $E_1$ is an equivalence relation. We show that ${R_1E_1[x] \subseteq E_1R_1[x]}$ for each $x \in Y_1$. If $x=c$ or $x=d$, then $E_1[x]=\{x\}$. Therefore, $R_1E_1[x]=R_1[x] \subseteq E_1R_1[x]$. On the other hand, if $x=a$ or $x=b$, then $E_1[x]=\{ a, b\}$ and $R_1[x]=\{a,b,d\}$ or $R_1[x]=\{b,d\}$. Thus, $R_1E_1[x]=\{a,b,d\} \subseteq E_1R_1[x]$. 
Since $\mathfrak{H}_1$ is finite, its topology is discrete, and hence $\mathfrak{H}_1$ is a descriptive $\ms$-frame. The proof that $\mathfrak{H}_2$ is a descriptive $\ms$-frame is similar.

\cref{fig:sk(f) not p-morph wrt E}(b) depicts the skeletons $\sk(\mathfrak{H}_1)$ and $\sk(\mathfrak{H}_2)$, where the black arrows represent the partial orders $R_i'$ and the red circles the $E_{Q_i'}$-equivalence classes for $i=1,2$.
Define ${f\colon\mathfrak{H}_1\to\mathfrak{H}_2}$ by 
\[
f(a)=u, \ f(b)=v, \mbox{ and } f(c)=f(d)=w.
\]
It is straightforward to check that $f$ is a $\msfrm$-morphism. However, $\sk(f)\colon\sk(\mathfrak{H}_1)\to\sk(\mathfrak{H}_2)$ is not a p-morphism with respect to $E_{Q'}$ because $\sk(f)(\pi_1(c))=\pi_2(w)$ and $\pi_2(w) E_{Q_2'} \pi_2(u)$, but $E_{Q_1'}[\pi_1(c)] = \{\pi_1(c)\}$, so there is no $x \in E_{Q_1'}[\pi_1(c)]$ such that $\sk(f)(x)=\pi_2(u)$.
\end{example}

We are ready to prove that $\sk \colon \msfrm \to \mipcfrm$ is dual to $\Op \colon \msa \to \mha$.

\begin{theorem}\label{thm:Op and sk}
The following diagram commutes up to natural isomorphism.
\[
\begin{tikzcd}[sep=huge]
\msa \arrow[d, "\Op"'] \arrow[r, shift left=1, "(-)_*"] & \msfrm \arrow[d, "\sk"] \arrow[l, shift left=1, "(-)^*"]  \\
\mha \arrow[r, shift left=1, "(-)_*"] & \mipcfrm \arrow[l, shift left=1, "(-)^*"]
\end{tikzcd}
\]
\end{theorem}

\begin{proof}
It is enough to prove that $(-)^* \circ \sk$ is naturally isomorphic to $\Op \circ (-)^*$. Let $\G=(Y,R,E)$ be a descriptive $\ms$-frame, $\sk(\G)=(X,R',Q')$ its skeleton, and $\pi\colon Y \to X$ the quotient
map.
Clearly, if $U$ is a clopen $R$-upset of $\sk(\G)$, then $\pi^{-1}[U]$ is a clopen $R$-upset of $\G$. In fact, $\pi^{-1} \colon \sk(\G)^* \to \Op(\G^*)$ is an isomorphism of Heyting algebras (see, e.g., \cite[Prop.~3.4.15]{Esa19}).
It remains to show that $\pi^{-1}$ commutes with $\forall$ and $\exists$. To simplify notation, let $H = \Op(\G^*)$. For $U \in \sk(\G)^*$, 
it follows from \cref{def:Op(B)} and the proof of \cref{thm: duality for ms4} that
\begin{align*}
\forall_H (\pi^{-1}[U]) &= Y \setminus R^{-1}[Y \setminus [Y \setminus E[Y \setminus \pi^{-1}[U]]]] = Y \setminus R^{-1}E[Y \setminus \pi^{-1}[U]]\\
 &= 
 Y \setminus Q^{-1}[Y \setminus \pi^{-1}[U]].
\end{align*}
Thus, by \cref{rem:ClopUp and Pf},
\begin{align*}
\pi^{-1}[\forall U] &= \pi^{-1}[X \setminus (Q')^{-1}[X \setminus U]] = Y \setminus \pi^{-1}[(Q')^{-1}[X \setminus U]]\\
&= Y \setminus Q^{-1}[\pi^{-1}[X \setminus U]] = Y \setminus Q^{-1}[Y \setminus \pi^{-1}[U]] = \forall_H (\pi^{-1}[U]).
\end{align*}
\cref{def:Op(B)}, \cref{rem:ClopUp and Pf}, and the proof of \cref{thm: duality for ms4} also yield 
\[
\exists U = Q'[U]  
\ \mbox{ and } \ \exists_H (\pi^{-1}[U]) = E[\pi^{-1}[U]].
\] 
Therefore,
\begin{align*}
\pi^{-1}[\exists U] = \pi^{-1}[Q'[U]] = Q[\pi^{-1}[U]] = ER[\pi^{-1}[U]] = E[\pi^{-1}[U]] = \exists_H (\pi^{-1}[U]),
\end{align*}
where the fourth equality follows from the fact that $\pi^{-1}[U]$ is an $R$-upset. Consequently, $\pi^{-1}$ is an $\mha$-isomorphism.

To see the naturality of this isomorphism, 
we need to show that $\pi_1^{-1} \circ \sk(f)^{-1} = f^{-1} \circ \pi_2^{-1}$ for every $\msfrm$-morphism $f \colon \G_1 \to \G_2$. For each $U \in \sk(\G_2)^*$ and $x \in Y_1$, the definition of $\rho(f)$ yields
\[
x \in \pi_1^{-1}[\sk(f)^{-1}[U]] \iff \sk(f)(\pi_1(x)) \in U \iff \pi_2(f(x)) \in U \iff x \in f^{-1}[\pi_2^{-1}[U]],
\]
and hence $\pi_1^{-1} \circ \sk(f)^{-1}$ and $f^{-1} \circ \pi_2^{-1}$ coincide on $U$.
\end{proof}

Obtaining an algebraic insight of $\tau$ and $\sigma$ requires investigating the realizability of monadic Heyting algebras as the algebras of open elements of $\ms$-algebras. Here the situation is more complicated because unlike the classic case of Heyting algebras (see, e.g., \cite[Sec.~IV.3]{RS63}),
it remains open whether each monadic Heyting algebra can be realized this way.
We will discuss this in detail in a forthcoming paper. 

\section{Failure of Blok--Esakia for $\mipc$} \label{sec: failure BE}

In this final section we show that $\tau$ and $\sigma$ are lattice homomorphisms, thus generalizing the result of Maksimova and Rybakov \cite{MR74eng} (see also \cite[Thm.~9.66]{CZ97}) to the monadic setting. On the other hand, we show that $\rho$ is neither a lattice homomorphism nor one-to-one. From this we derive our main result, that $\sigma$ is not an isomorphism, and hence that the Blok--Esakia Theorem does not extend to the monadic setting.

\begin{lemma}\plabel{lem:property subdir irr}
\hfill\begin{enumerate}
\item\label[lem:property subdir irr]{lem:property subdir irr:item1} If $\A\in\mha$ is subdirectly irreducible, then $\forall a_1 \vee \forall a_2 =1$ implies $a_1=1$ or $a_2=1$ for any $a_1,a_2 \in A$.
\item\label[lem:property subdir irr]{lem:property subdir irr:item2} If $\B\in\msa$ is subdirectly irreducible, then $\bbox b_1 \vee \bbox b_2 =1$ implies $b_1=1$ or $b_2=1$ for any $b_1,b_2 \in B$.
\end{enumerate}
\end{lemma}

\begin{proof}
\eqref{lem:property subdir irr:item1} follows from \cite[Thm.~2.11]{Bez98b} and \eqref{lem:property subdir irr:item2} 
is proved similarly.
\end{proof}

Following \cite{MR74eng}, for two formulas $\varphi$ and $\psi$, we write $\varphi \vee' \psi$ for $\varphi \vee \psi'$, where $\psi'$ is obtained by substituting the variables in $\psi$ that occur in $\varphi$ with fresh variables, so that $\varphi$ and $\psi'$ have no variables in common. 

\begin{lemma}\label{lem:description intersection logics}
Let $\Gamma_1, \Gamma_2$ be sets of formulas in 
$\mathcal{L}_{\Box \forall}$ and $\M_i=\ms+\Gamma_i$ for $i=1,2$. Then
\[
\M_1 \cap \M_2 = \ms+\{\bbox\gamma_1 \vee' \bbox\gamma_2 : \gamma_1 \in \Gamma_1, \ \gamma_2 \in \Gamma_2\}.
\]
\end{lemma}

\begin{proof}
It is clear that
\[
\{\bbox\gamma_1 \vee' \bbox\gamma_2 : \gamma_1 \in \Gamma_1, \ \gamma_2 \in \Gamma_2\} \subseteq \M_1 \cap \M_2.
\]
We prove the other inclusion by showing that if a subdirectly irreducible $\ms$-algebra $\B$ validates $\bbox\gamma_1 \vee' \bbox\gamma_2$ for every $\gamma_1 \in \Gamma_1$ and $\gamma_2 \in \Gamma_2$, then it validates $\M_1 \cap \M_2$. We argue by contrapositive. Suppose that $\B \nvDash \M_1 \cap \M_2$. Since $\M_1\cap \M_2\subseteq \M_1,\M_2$, we obtain that $\B \nvDash \M_1,\M_2$. Therefore, 
there are $\gamma_1 \in \Gamma_1$ and $\gamma_2 \in \Gamma_2$ and two valuations $v_1,v_2$ on $\B$ such that $v_1(\gamma_1) \neq 1$ and $v_2(\gamma_2) \neq 1$. 
Let $p_1, \dots, p_n$ be the variables occurring in $\gamma_1$ and $q_1, \dots, q_m$ those occurring in $\gamma_2$. We let $q_1', \dots, q_m'$ be the variables that substitute $q_1, \dots, q_m$ to obtain $\gamma_1 \vee' \gamma_2$.
Define a valuation $v_2'$ on $\B$ by setting $v_2'(q_i')=v_2(q_i)$ for $i = 1, \dots, m$, and $v_2'(p)=v_2(p)$ for the remaining variables.  
Then $v_2'(\gamma_2')=v_2(\gamma_2)$.
Let $v$ be a new valuation 
that coincides with $v_1$ on the variables occurring in $\gamma_1$ and with $v_2'$ on the variables occurring in $\gamma_2'$. Then $v(\gamma_1) \neq 1$ and $v(\gamma_2') \neq 1$. Since $\B$ is subdirectly irreducible, 
\[
v(\bbox\gamma_1 \vee' \bbox\gamma_2) = \bbox v(\gamma_1) \vee \bbox v(\gamma_2') \neq 1
\]
by \cref{lem:property subdir irr:item2},
and hence $\B \nvDash \bbox\gamma_1 \vee' \bbox\gamma_2$. Thus, 
\[
\M_1 \cap \M_2 \subseteq \{\bbox\gamma_1 \vee' \bbox\gamma_2 : \gamma_1 \in \Gamma_1, \ \gamma_2 \in \Gamma_2\},
\]
concluding the proof.
\end{proof}

\begin{theorem}
$\tau$ and $\sigma$ are lattice homomorphisms.
\end{theorem}

\begin{proof}
By \cref{prop:rho tau sigma adjoints:item3}, $\tau$ and $\sigma$ preserve arbitrary joins. Thus, it suffices to show that they preserve binary meets. We only prove it for $\tau$ because the proof for $\sigma$ is similar. Let $\L_1,\L_2 \in \Lambda(\mipc)$. Since $\tau$ preserves $\subseteq$, we have 
$\tau(\L_1 \cap \L_2) \subseteq \tau \L_1  \cap \tau \L_2 $. 
For the other inclusion, by \cref{lem:description intersection logics} it is sufficient to show that $\bbox\varphi_1^t \vee' \bbox\varphi_2^t \in \tau(\L_1 \cap \L_2)$ for every $\varphi_1 \in \L_1$ and $\varphi_2 \in \L_2$. By the definition of the G\"odel translation, $\bbox\varphi_1^t \vee' \bbox\varphi_2^t =(\forall \varphi_1 \vee' \forall \varphi_2)^t$. Thus, $\bbox\varphi_1^t \vee' \bbox\varphi_2^t \in \tau(\L_1 \cap \L_2)$ because $\forall \varphi_1 \vee' \forall \varphi_2 \in \L_1 \cap \L_2$.
\end{proof}

\begin{remark}\plabel{rem:open problems}
\hfill\begin{enumerate}
\item\label[rem:open problems]{rem:open problems:item1}
Whether $\tau$ and $\sigma$ are complete lattice homomorphisms remains open.
\item\label[rem:open problems]{rem:open problems:item2}
Another open problem is the surjectivity of $\rho$. Equivalently, the question of whether every extension of $\mipc$ has a modal companion remains open. In turn,
this is equivalent to determining whether $\tau \L$ is a modal companion of $\L$ for every $\L \in \Lambda(\mipc)$. Indeed, if $\L$ has a modal companion, then $\tau \L$ must be the least such. In \cite{BC24b} we will show that Kripke completeness is a sufficient condition for $\L$ to have a modal companion.
\end{enumerate}
\end{remark}

We now turn our attention to $\rho$.
Since $\mha$ and $\msa$ are congruence-distributive varieties, we will freely use 
J\'onsson's Lemma and especially its corollary that
if a congruence-distributive variety $\mathbb{V}$ is generated by a finite algebra $\A$, then subdirectly irreducible algebras in $\mathbb{V}$ are in 
$\H\S(\A)$ (see, e.g., \cite[Cor.~IV.6.10]{BS81}). We will also utilize that both $\mha$ and $\msa$ have the CEP, 
and hence that $\H\S=\S\H$.

The next lemma is a generalization of a similar result for $\ha$ and $\sa$ (see, e.g., \cite[Thm.~3.4.16]{Esa19}). 

\begin{lemma}\plabel{lem:congr and closed Q-up}
Let $\A \in \mha$ and $\B \in \msa$. There are inclusion-reversing bijections between
\begin{enumerate}
\item\label[lem:congr and closed Q-up]{lem:congr and closed Q-up:item1} the sets of monadic filters of $\A$ and closed $Q$-upsets of $\A_*$, and
\item\label[lem:congr and closed Q-up]{lem:congr and closed Q-up:item2} the sets of monadic $\Box$-filters of $\B$ and closed $Q$-upsets of $\B_*$.
\end{enumerate}
\end{lemma}

\begin{proof}[Sketch of proof.]
The proof of \eqref{lem:congr and closed Q-up:item1} can be found in \cite[Thm.~18]{Bez99} and \eqref{lem:congr and closed Q-up:item2} is proved similarly (see, e.g., \cite[Thm.~3.4]{BM24}).
\end{proof}

\begin{remark}\label{rem:quotients and closed Q-upsets}
Let $\A \in \mha$, $\theta$ be a congruence on $\A$, and $Z$ the corresponding closed $Q$-upset of $\A_*=(X,R,Q)$. 
The quotient $\A/\theta$ is then dual to the descriptive $\mipc$-frame $(Z,R_{|Z},Q_{|Z})$ obtained by restricting $R$ and $Q$ to $Z$. A similar correspondence holds for quotients of $\ms$-algebras and closed $Q$-upsets of their duals.
\end{remark}

\begin{definition}
A descriptive $\mipc$-frame $(X,R,Q)$ is \emph{strongly $Q$-rooted} if there is $x \in X$ such that $Q[x]=X$ and $E_Q[x]$ is clopen. 
\end{definition}

Strongly $Q$-rooted descriptive $\ms$-frames are defined similarly. 
As a consequence of \cref{lem:congr and closed Q-up}, we have the following dual characterization of subdirectly irreducible algebras in $\mha$ and $\msa$, which generalizes a similar characterization of subdirectly irreducible algebras in $\ha$ and $\sa$ (see, e.g., \cite[Prop.~A.1.2]{Esa19}). The proof of \eqref{prop:subdirectly irr dually:item1} can be found in \cite[Thm.~24]{Bez99}, and \eqref{lem:congr and closed Q-up:item2} is proved similarly (see, e.g., \cite[Thm.~3.5]{BM24}). 

\needspace{5\baselineskip}
\begin{lemma}\plabel{prop:subdirectly irr dually}
\hfill\begin{enumerate}
\item\label[prop:subdirectly irr dually]{prop:subdirectly irr dually:item1} $\A \in \mha$ is subdirectly irreducible iff $\A_*$ is strongly $Q$-rooted.
\item\label[prop:subdirectly irr dually]{prop:subdirectly irr dually:item2} $\B \in \msa$ is subdirectly irreducible iff $\B_*$ is strongly $Q$-rooted.
\end{enumerate}
\end{lemma}

Let $h \colon \B_1 \to \B_2$ be a homomorphism of $\sfour$-algebras. It is well known that $h$ is one-to-one iff $h_*$ is onto (see, e.g., \cite[Lem.~3.3.13]{Esa19}), and that the same holds for Heyting algebra homomorphisms. As an immediate consequence, we obtain:

\begin{lemma}\plabel{prop:1-1 and onto}
\hfill\begin{enumerate}
\item\label[prop:1-1 and onto]{prop:1-1 and onto:item1} A homomorphism $f$ of monadic Heyting algebras is one-to-one iff $f_*$ is onto.
\item\label[prop:1-1 and onto]{prop:1-1 and onto:item2} A homomorphism $g$ of $\ms$-algebras is one-to-one iff $g_*$ is onto.  
\end{enumerate}
\end{lemma}

The following dual characterization of the composition of the operators $\H$ and $\S$ is an immediate consequence of \cref{lem:congr and closed Q-up}, \cref{prop:1-1 and onto}, and the fact that both $\mha$ and $\msa$ have the CEP.

\begin{lemma}\plabel{prop:HS dual}
Let $\F_1,\F_2\in\mipcfrm$ and $\G_1,\G_2\in\msfrm$.
\begin{enumerate}
\item\label[prop:HS dual]{prop:HS dual:item1} 
$\F_2^*\in \H\S(\F_1^*)$ $\iff$ $\F_2^*\in \S\H(\F_1^*)$ $\iff$
there is an onto $\mipcfrm$-morphism from a closed
$Q_1$-upset of $\F_1$ to $\F_2$.
\item\label[prop:HS dual]{prop:HS dual:item2} 
$\G_2^*\in \H\S(\G_1^*)$ $\iff$
$\G_2^*\in \S\H(\G_1^*)$ $\iff$
there is an onto $\msfrm$-morphism from a closed
$Q_1$-upset of $\G_1$ to $\G_2$.
\end{enumerate}
\end{lemma}

\begin{figure}
\begin{tikzpicture}[-{Latex[width=1mm]}]
\coordinate (B1) at (-2,-1.5);
\coordinate (ML1) at (-3.5,0);
\coordinate (MR1) at (-2,0);
\coordinate (T1) at (-2,1.5);
\coordinate (B2) at (3,-1.5);
\coordinate (M2) at (3,0);
\coordinate (T2) at (3,1.5);
\fill (B1) circle(2pt);
\fill (T1) circle(2pt);
\fill (ML1) circle(2pt);
\fill (MR1) circle(2pt);
\fill (B2) circle(2pt);
\fill (M2) circle(2pt);
\fill (T2) circle(2pt);
\draw (ML1) -- (T1);
\draw (MR1) -- (T1);
\draw (M2) -- (T2);
\draw (B1) -- (ML1);
\draw (B1) -- (MR1);
\draw (B2) -- (M2);
\clustertwo{MR1}{T1}{1.6}{1};
\clusterone{ML1}{1.2};
\clusterone{B1}{1.2};
\clusterone{B2}{1.2};
\clustertwo{M2}{T2}{1.6}{1};
\node at ([shift={(0:0.6)}]B1) {$u$};
\node at ([shift={(0:0.6)}]MR1) {$v$};
\node at ([shift={(0:0.6)}]ML1) {$w$};
\node at ([shift={(0:0.65)}]T1) {$z$};
\node at ([shift={(0:0.6)}]B2) {$a$};
\node at ([shift={(0:0.6)}]M2) {$b$};
\node at ([shift={(0:0.6)}]T2) {$c$};
\node at (-2,-2.4) {$\mathfrak{K}_1$};
\node at (3.1,-2.4) {$\mathfrak{K}_2$};
\end{tikzpicture}
\caption{The frames $\mathfrak{K}_1$ and $\mathfrak{K}_2$.}\label{fig: G1 and G2}
\end{figure}

\begin{definition}\label{def:G1 and G2}
Let $\mathfrak{K}_1=(Y_1,R_1,E_1)$ and $\mathfrak{K}_2=(Y_2,R_2,E_2)$ be the descriptive $\ms$-frames depicted in \cref{fig: G1 and G2}, where the black arrows 
represent the partial orders $R_i$ and the red circles the $E_i$-equivalence classes ($i=1,2$).
\end{definition}

\begin{remark}
An argument similar to \cref{ex:sk(f) not p-morph wrt E} gives that $\mathfrak{K}_1$ and $\mathfrak{K}_2$ are indeed descriptive $\ms$-frames.
\end{remark}

Let $\B_1=\mathfrak{K}_1^*$ and $\B_2=\mathfrak{K}_2^*$. Since both $\mathfrak{K}_1$ and $\mathfrak{K}_2$ are finite, their topologies are discrete, and so 
$\B_1$ and $\B_2$ are the powersets of $Y_1$ and $Y_2$, respectively. We also define $\mathbb{V}_1\coloneqq\V(\B_1)$ and $\mathbb{V}_2\coloneqq\V(\B_2)$. Because each 
$R_i$ is a partial order, 
each $\B_i$ is a finite $\mgrz$-algebra (see, e.g., \cite[Cor.~3.5.10]{Esa19}).

\begin{lemma}\plabel{prop:morph between G1 and G2}
\hfill\begin{enumerate}
\item\label[prop:morph between G1 and G2]{prop:morph between G1 and G2:item1} There is an onto $\mipcfrm$-morphism from $\sk(\mathfrak{K}_1)$ to $\sk(\mathfrak{K}_2)$.
\item\label[prop:morph between G1 and G2]{prop:morph between G1 and G2:item2} There is no onto $\msfrm$-morphism from a $Q_1$-upset of $\mathfrak{K}_1$ to $\mathfrak{K}_2$.
\end{enumerate}
\end{lemma}

\begin{proof}
\eqref{prop:morph between G1 and G2:item1}. 
Since $R_i$ is a partial order for $i=1,2$, we have that
$\pi_i \colon \mathfrak{K}_i \to \sk(\mathfrak{K}_i)$ is a bijection that preserves and reflects $R_i$ and $Q_i$. Thus, we may identify $\rho(\mathfrak{K}_i)$ with $\mathfrak{K}_i$.
Define $f \colon \sk(\mathfrak{K}_1) \to \sk(\mathfrak{K}_2)$ 
by
\[
f(u)=a, \quad f(v)=b, \quad \text{and} \quad f(w)=f(z)=c.
\]
Clearly $f$ is onto, and it is straightforward to see that $f$ is a p-morphism with respect to $R$ and $Q$. The following calculations show that $f$ is also a weak p-morphism with respect to $Q^{-1}$:
\begin{align*}
Q_2^{-1}[f(u)] & = Q_2^{-1}[a] = \{ a \}  = R_2^{-1}[a] = R_2^{-1}[f(u)] = R_2^{-1}fQ_1^{-1}[u], \\
Q_2^{-1}[f(v)] & = Q_2^{-1}[b] = Y_2 = R_2^{-1}[Y_2] = R_2^{-1}f[Y_1] = R_2^{-1}fQ_1^{-1}[v], \\
Q_2^{-1}[f(w)] & = Q_2^{-1}[c] = Y_2 = R_2^{-1}[\{ a, c \}] = R_2^{-1}f[\{u, w \}] = R_2^{-1}fQ_1^{-1}[w], \\
Q_2^{-1}[f(z)] & = Q_2^{-1}[c] = Y_2 = R_2^{-1}[Y_2] = R_2^{-1}f[Y_1] = R_2^{-1}fQ_1^{-1}[z].
\end{align*}
Thus, 
$f$ is a $\mipcfrm$-morphism.

\eqref{prop:morph between G1 and G2:item2}. 
Suppose there is a $Q_1$-upset $U$ of $\mathfrak{K}_1$ and an onto $\msfrm$-morphism $g \colon U \to \mathfrak{K}_2$.
Since $g$ is onto, there is $x \in U$ such that $g(x)=a$. Because $g$ is a p-morphism with respect to $R$, we have $g[R_1[x]] =R_2[g(x)]= R_2[a]$. Since $R_2[a]$ has $3$ elements, $R_1[x]$ must have at least $3$ elements. Therefore, $x=u$, and hence $u \in U$. This implies that $U=Y_1$.
Because the cardinality of $R_1[w]$ is $2$,
the cardinality of $R_2[g(w)]=g[R_1[w]]$ is at most $2$. Thus, $g(w)\in \{b,c\}$. That $g$ is a p-morphism with respect to $E$ implies that $g[E_1[w]]=E_2[g[w]]=
\{b,c\}$. Consequently, $E_1[w]$ must have at least $2$ elements, which is a contradiction, proving that such a $g$ does not exist. 
\end{proof}

\begin{lemma}\label{lem:is poset then sk iso implies iso}
Let $\G_1,\G_2$ be two partially ordered descriptive $\ms$-frames. 
If $\sk(\G_1)$ and $\sk(\G_2)$ are isomorphic in $\mipcfrm$, then $\G_1$ and $\G_2$ are isomorphic in $\msfrm$.
\end{lemma}

\begin{proof}
Since $\G_i$ is partially ordered for $i=1,2$, we have that $\pi_i\colon\G_i\to\rho(\G_i)$ is a bijection that preserves and reflects the relations on $\G_i$. Because a $\mipcfrm$-isomorphism ${g\colon\sk(\G_1)\to\sk(\G_2)}$ is a bijection that preserves and reflects the relations on $\sk(\G_1)$ (see \cref{rem:isos mipc}), it gives rise to a bijection $f\colon\G_1\to\G_2$ that preserves and reflects the relations on $\G_1$. Thus, $f$ is a $\msfrm$-isomorphism (see \cref{rem:isos}).
\end{proof}

\begin{lemma}\plabel{prop:Ai and Bi}
\hfill\begin{enumerate}
\item\label[prop:Ai and Bi]{prop:Ai and Bi:item1} $\Op(\B_2)$ embeds into $\Op(\B_1)$.
\item\label[prop:Ai and Bi]{prop:Ai and Bi:item2} $\B_2 \notin \H\S(\B_1)$.
\item\label[prop:Ai and Bi]{prop:Ai and Bi:item3} $\Op(\B_2) \notin \Op(\mathbb{V}_1)$.
\end{enumerate}
\end{lemma}

\begin{proof}
\eqref{prop:Ai and Bi:item1}. 
This follows
from \cref{prop:1-1 and onto:item1,prop:morph between G1 and G2:item1}.

\eqref{prop:Ai and Bi:item2}.
This follows from \cref{prop:HS dual:item2,prop:morph between G1 and G2:item2}.

\eqref{prop:Ai and Bi:item3}. 
Suppose $\Op(\B_2) \in \Op(\mathbb{V}_1)$. Then there is $\B \in \mathbb{V}_1$
such that $\Op(\B_2)=\Op(\B)$. 
Since $Q_2[a]=Y_2$, it follows from \cref{prop:subdirectly irr dually:item2} that $\B_2$ is subdirectly irreducible, so $\Op(\B_2)=\Op(\B)$ is subdirectly irreducible, and hence $\B$ is subdirectly irreducible by \cref{cor:B sub irr iff O(B) sub irr}. Therefore, since $\B$ belongs to the variety generated by $\B_1$, J\'onsson's Lemma yields that $\B \in \H\S(\B_1)$.
By \cref{prop:HS dual:item2}, there are a $Q_1$-upset $U$ of $\mathfrak{K}_1$ and a $\msfrm$-morphism $f$ from $U$ onto the dual $\G=(Y,R,E)$ of $\B$. 
Because $R_1$ is a partial order, $\B_1$ is a finite $\mgrz$-algebra. Hence, $\B$ is a $\mgrz$-algebra since $\B \in \H\S(\B_1)$. Therefore, $R$ is a partial order.
Because $\Op(\B_2)=\Op(\B)$, the frames $\sk(\mathfrak{K}_2)$ and $\sk(\G)$ are isomorphic in $\mipcfrm$.
Thus, \cref{lem:is poset then sk iso implies iso} yields that $\mathfrak{K}_2$ and $\G$ are isomorphic in $\msfrm$, and so there is a $\msfrm$-morphism $f$ from $U$ onto $\mathfrak{K}_2$. This contradicts \cref{prop:morph between G1 and G2:item2}.
\end{proof}

\cref{prop:SOpV=VOp,thm:int frag SrhoV} show that the operator $\S\Op$ plays the same role for varieties of $\ms$-algebras  
as $\Op$ does for varieties of $\sfour$-algebras. 
The following theorem yields that $\S\Op$ is not well behaved, already when restricted to varieties of $\mgrz$-algebras. 

\begin{theorem}\plabel{thm:SO not well behaved}
\hfill\begin{enumerate}
\item\label[thm:SO not well behaved]{thm:SO not well behaved:item1} $\mathbb{V}_1$ is a variety of $\mgrz$-algebras such that $\Op(\mathbb{V}_1)$ is not a variety of monadic Heyting algebras. 
\item\label[thm:SO not well behaved]{thm:SO not well behaved:item2} $\S$ and $\Op$ do not commute.
\item\label[thm:SO not well behaved]{thm:SO not well behaved:item3} $\S\Op$ does not commute with binary intersections.
\item\label[thm:SO not well behaved]{thm:SO not well behaved:item4} $\S\Op$ is not one-to-one.
\end{enumerate}
\end{theorem}

\begin{proof}
\eqref{thm:SO not well behaved:item1}.
Let $\mgrza$ be the variety of $\mgrz$-algebras. Since $\B_1 \in \mgrza$, we have that $\mathbb{V}_1$ is a subvariety of $\mgrza$.
Because $\Op(\B_1) \in \Op(\mathbb{V}_1)$, it follows from \cref{prop:Ai and Bi:item1} that $\Op(\B_2) \in \S \Op(\mathbb{V}_1)$. By \cref{prop:Ai and Bi:item3}, $\Op(\B_2) \notin \Op(\mathbb{V}_1)$. Consequently, $\S \Op(\mathbb{V}_1) \neq \Op(\mathbb{V}_1)$, so $\Op(\mathbb{V}_1)$ is not closed under $\S$, and hence is not a variety. 

\eqref{thm:SO not well behaved:item2}.
As we saw in the proof of \eqref{thm:SO not well behaved:item1},  $\S \Op(\mathbb{V}_1) \neq \Op (\mathbb{V}_1) = \Op \S (\mathbb{V}_1)$. Thus, $\S$ and $\Op$ do not commute. 

\eqref{thm:SO not well behaved:item3}. 
We show that $\S\Op(\mathbb{V}_1 \cap \mathbb{V}_2) \neq \S\Op(\mathbb{V}_1) \cap \S\Op(\mathbb{V}_2)$.
By \cref{prop:Ai and Bi:item1}, $\Op(\B_2) \in \S\Op(\B_1)$, so $\Op(\B_2) \in \S\Op(\mathbb{V}_1) \cap \S\Op(\mathbb{V}_2)$. It remains to prove that $\Op(\B_2) \notin \S\Op(\mathbb{V}_1 \cap \mathbb{V}_2)$. By J\'onsson's Lemma, every subdirectly irreducible algebra in $\mathbb{V}_1 \cap \mathbb{V}_2$ is in $\H\S(\B_1)\cap\H\S(\B_2)$. Therefore, $\mathbb{V}_1 \cap \mathbb{V}_2 = \V(\H\S(\B_1)\cap\H\S(\B_2))$. Thus, by \cref{prop:SOpV=VOp:item1}, 
\[
\S\Op(\mathbb{V}_1 \cap \mathbb{V}_2) = \S\Op(\V(\H\S(\B_1)\cap\H\S(\B_2)))=\V(\Op(\H\S(\B_1)\cap\H\S(\B_2))).
\]
Using J\'onsson's Lemma again, if $\Op(\B_2) \in \V(\Op(\H\S(\B_1)\cap\H\S(\B_2)))$, then $\Op(\B_2) \in \H\S(\Op(\H\S(\B_1)\cap\H\S(\B_2)))$. Consequently, there exists
$\B \in \H\S(\B_1)\cap \H\S(\B_2)$ such that $\Op(\B_2) \in \H\S(\Op(\B))$.
Let $\G \in \msfrm$ be the dual of $\B$. Since $\B \in \H\S(\B_2)$, \cref{prop:HS dual:item2} yields that the cardinality $\lvert \G \rvert$ 
is less than or equal to $\lvert \mathfrak{K}_2 \rvert=3$. Because $\Op(\B_2) \in \H\S(\Op(\B))$, from \cref{prop:HS dual:item1} it follows that $3=\lvert\sk(\mathfrak{K}_2)\rvert \le \lvert\sk(\G)\rvert$. Therefore, $\lvert \G \rvert = \lvert \mathfrak{K}_2 \rvert = 3$, so \cref{prop:HS dual:item2} 
implies that $\G \cong \mathfrak{K}_2$, and hence $\B \cong \B_2$. Thus, from $\B \in \H\S(\B_1)$ it follows that $\B_2 \in \H\S(\B_1)$, which contradicts \cref{prop:Ai and Bi:item2}.

\eqref{thm:SO not well behaved:item4}. 
Let $\mathbb{V}$ be the join of $\mathbb{V}_1$ and $\mathbb{V}_2$. We show that $\S\Op(\mathbb{V})=\S\Op(\mathbb{V}_1)$ but $\mathbb{V} \neq \mathbb{V}_1$. We have
\begin{align*}
\S\Op(\mathbb{V}) &= \S\Op(\V(\{\B_1,\B_2\}))=\V(\{\Op(\B_1),\Op(\B_2)\})\\
&=\V(\Op(\B_1))=\S\Op(\V(\B_1))=\S\Op(\mathbb{V}_1),
\end{align*}
where the first and last equalities follow from the definitions of $\mathbb{V}$ and $\mathbb{V}_1$, the second and fourth from \cref{prop:SOpV=VOp:item1}, and the third is a consequence of \cref{prop:Ai and Bi:item1}.
J\'onsson's Lemma together with \cref{prop:Ai and Bi:item2} yields that $\B_2 \notin \V(\B_1)$. Thus, $\mathbb{V}=\V(\{\B_1,\B_2\})\neq\V(\B_1)=\mathbb{V}_1$. 
\end{proof}

\begin{remark}
\begin{figure}[!h]
\begin{tikzpicture}[-{Latex[width=1mm]}]
\coordinate (M3) at (0,0);
\coordinate (T3) at (0,1.5);
\fill (M3) circle(2pt);
\fill (T3) circle(2pt);
\draw (M3) -- (T3);
\clustertwo{M3}{T3}{1.6}{1};
\node at (0.1,-1.2) {$\mathfrak{K}_3$};
\coordinate (M4) at (3,0);
\coordinate (T4) at (3,1.5);
\fill (M4) circle(2pt);
\fill (T4) circle(2pt);
\draw (M4) -- (T4);
\clusterone{M4}{1.2};
\clusterone{T4}{1.2};
\node at (3.1,-1.2) {$\mathfrak{K}_4$};
\coordinate (T5) at (6,1.5);
\fill (T5) circle(2pt);
\clusterone{T5}{1.2};
\node at (6.1,-1.2) {$\mathfrak{K}_5$};
\end{tikzpicture}
\caption{The frames $\mathfrak{K}_3$, $\mathfrak{K}_4$, and $\mathfrak{K}_5$.}
\label{fig3}
\end{figure}

As we saw in \cref{thm:SO not well behaved:item1},
 $\Op(\mathbb{V}_1)$ is not a variety. On the other hand, $\Op(\mathbb{V}_2)$ is a variety. This can be seen as follows.
For $\G \in \msfrm$, it is straightforward to see that there is an onto $\msfrm$-morphism from a closed $Q_2$-upset of $\mathfrak{K}_2$ to $\G$ iff $\G \cong \mathfrak{K}_i$ for $i=2,\dots, 5$, where $\mathfrak{K}_3, \mathfrak{K}_4, \mathfrak{K}_5$ are shown in Fig.~\ref{fig3}.
 Thus, by 
\cref{prop:HS dual:item2},
 $\G^* \in \H\S(\B_2)$ iff $\G \cong \mathfrak{K}_i$ for $i=2,\dots, 5$.

Similarly, for $\F \in \mipcfrm$, it is straightforward to see that there is an onto $\mipcfrm$-morphism from a closed $Q_2'$-upset of $\sk(\mathfrak{K}_2)$ to $\F$ iff $\F \cong \sk(\mathfrak{K}_i)$ for $i=2,\dots, 5$. Thus, by 
\cref{prop:HS dual:item2}, 
$\F^* \in \H\S\Op(\B_2)$ iff $\F \cong \sk(\mathfrak{K}_i)$ for $i=2,\dots, 5$.

J\'onsson's Lemma yields that the subdirectly irreducible $\ms$-algebras in $\mathbb{V}_2$ are exactly the ones isomorphic to $\mathfrak{K}_i^*$ for $i=2,\dots, 5$. Since $\mathbb{V}_2 = \V(\mathfrak{B}_2)$, by 
\cref{prop:SOpV=VOp} we have
\[
\V(\Op(\mathbb{V}_2)) = \S\Op(\mathbb{V}_2) = \S\Op(\V(\B_2)) = \V(\Op(\B_2)).
\]
So, J\'onsson's Lemma implies that the subdirectly irreducible monadic Heyting algebras in $\V(\Op(\mathbb{V}_2))$ are exactly the ones isomorphic to $\sk(\mathfrak{K}_i)^*$ for $i=2,\dots, 5$.

Let $\A \in \V(\Op(\mathbb{V}_2))$ and $\A_*=(X,R,Q)$.
For each $x \in X$ we have that $Q[x] \cong \sk(\mathfrak{K}_i)$ for $i=2,\dots, 5$.
Therefore, for each $x \in X$, the $R$-upset $R[x]$ is a chain of at most $3$ elements, and $E[x]$ is either a singleton or a $2$-element chain whose top element is maximal in $X$. From this we can derive that $E_Q[U]$ is clopen for each clopen $U$, and hence that $E_Q$ is a continuous relation on $X$. 
Thus, $(X,R,E_Q)$ is a descriptive $\ms$-frame. 

Let $\B$ be the $\ms$-algebra dual to $(X,R,E_Q)$. Then $\A \cong \Op(\B)$ by \cref{thm:Op and sk}. For every $x \in X$ we have that $Q[x]$ is isomorphic to $\sk(\mathfrak{K}_i)$ for some $i=2,\dots, 5$. Since $\sk(\mathfrak{K}_i) \cong \mathfrak{K}_i$ for every $i =2, \dots, 5$, it follows that each $Q[x]$ is isomorphic to $\mathfrak{K}_i$ for some $i=2,\dots, 5$. This means that $\B$ is a subdirect product of a family of $\ms$-algebras each isomorphic to $\mathfrak{K}_i^*$ for some $i=2,\dots, 5$. We observed above that $\mathfrak{K}_i^* \in \H\S(\B_2)$ for every $i=2,\dots, 5$. Thus, $\B \in \V(\B_2)=\mathbb{V}_2$. Consequently, $\A \in \Op(\mathbb{V}_2)$. This proves that $\V(\Op(\mathbb{V}_2)) \subseteq \Op(\mathbb{V}_2)$, and so $\Op(\mathbb{V}_2)$ is a variety.
\end{remark}

By putting \cref{thm: lattice iso,thm: lattice iso ms,thm:int frag SrhoV,thm:SO not well behaved} together, we obtain:

\begin{theorem}\plabel{cor: no BE}
\hfill\begin{enumerate}
\item\label[cor: no BE]{cor: no BE:item1} $\rho \colon \Lambda(\mgrz) \to \Lambda(\mipc)$ is not a lattice homomorphism. 
\item\label[cor: no BE]{cor: no BE:item2} $\rho \colon \Lambda(\mgrz) \to \Lambda(\mipc)$ is not one-to-one.
\end{enumerate}
\end{theorem}

\begin{proof}
\eqref{cor: no BE:item1}.
Let $\Lambda(\mgrza)$ be the complete lattice of subvarieties of $\mgrza$. It follows from \cref{thm: lattice iso ms:item2} that $\Lambda(\mgrz)$ is dually isomorphic to $\Lambda(\mgrza)$. By \cref{thm:SO not well behaved:item3}, $\S\Op\colon \Lambda(\mgrza) \to \Lambda(\mha)$ is not a lattice homomorphism. Therefore, neither is $\rho$ by Theorems~\ref{thm: lattice iso} and \ref{thm:int frag SrhoV}\eqref{thm:int frag SrhoV:item1}. 

\eqref{cor: no BE:item2}. This is proved similarly, but uses \cref{thm:SO not well behaved:item4}.
\end{proof}

The previous theorem immediately yields:

\begin{corollary}
$\rho \colon \Lambda(\ms) \to \Lambda(\mipc)$ is neither a lattice homomorphism nor one-to-one. 
\end{corollary}

We are ready to prove that the Blok--Esakia Theorem fails in the monadic setting.

\begin{theorem}\label{thm:sigma not onto}
$\sigma \colon \Lambda(\mipc) \to \Lambda(\mgrz)$ is not onto, hence is not an isomorphism.
\end{theorem}

\begin{proof}
By \cref{prop:rho tau sigma adjoints:item2}, $\sigma$ is left adjoint to $\rho \colon \Lambda(\mgrz) \to \Lambda(\mipc)$. Thus, $\sigma \rho \sigma = \sigma$. If $\sigma$ were onto, then $\sigma \rho$ would be the identity on $\Lambda(\mgrz)$, and hence $\rho$ would be one-to-one. This contradicts \cref{cor: no BE:item2}.
\end{proof}

In this paper we showed that the Blok--Esakia isomorphism does not extend to the monadic setting. We conclude by outlining several interesting directions for future research. 
\begin{itemize}
\item It remains open whether the two lattices $\Lambda(\mipc)$ and $\Lambda(\mgrz)$ are isomorphic. Our expectation is that the answer is negative.
\item It is also open whether each monadic Heyting algebra can be realized as the algebra of open elements of some monadic $\mathsf{S4}$-algebra (see the end of \cref{sec:Op}).
\item A related open problem is whether $\rho \colon  \Lambda(\ms) \to \Lambda(\mipc)$ is surjective. In other words, it remains open whether every extension of $\mipc$ has a modal companion (see \cref{rem:open problems:item2}).
\item In addition, it is open whether $\tau$ and $\sigma$ are complete lattice homomorphisms (see \cref{rem:open problems:item1}).
\end{itemize}

\section*{Acknowledgements}

We would like to thank the referees for careful reading and useful comments which have improved the presentation.

\end{document}